\documentclass[reqno, 12pt,a4letter]{amsart}
\usepackage{amsmath,amsxtra,amssymb,latexsym, amscd,amsthm}
\usepackage{graphicx}
\usepackage[utf8]{inputenc}
\usepackage[mathscr]{euscript}
\usepackage{mathrsfs}
\usepackage[english]{babel}
\usepackage{color}
\usepackage{enumerate}
\usepackage{graphicx} 
\usepackage[T1]{fontenc}

\newtheorem {theorem}{Theorem}[section]
\newtheorem {corollary}{Corollary}[section]
\newtheorem {proposition}{Proposition}[section]
\newtheorem {lemma}{Lemma}[section]

\theoremstyle{definition}
\newtheorem{definition}{Definition}[section]
\newtheorem {example}{Example}[section]
\newtheorem{remark}{Remark}[section]

\newcommand{\rank}{{\rm rank}}
\newcommand{\co}{{\rm conv}}
\renewcommand{\span}{{\rm span}}
\newcommand{\dis}{\displaystyle}

\def\ow{o\kern-.42em\raise.82ex\hbox{
   \vrule width .12em height .0ex depth .075ex \kern-0.16em \char'56}\kern-.07em}
\def\OW{o\kern-.460em\raise1.36ex\hbox{
\vrule width .13em height .0ex depth .075ex \kern-0.16em
\char'56}\kern-.07em}
\def\DD{D\kern-.7em\raise0.4ex\hbox{\char '55}\kern.33em}

\title{Global {{\L}}ojasiewicz inequalities on comparing the rate of growth of polynomial functions}

\author[S. T. \DD inh]{S\~i-Ti\d{\^e}p \DD inh$^\dag$}
\address{Institute of Mathematics, VAST, 18, Hoang Quoc Viet Road, Cau Giay District 10307, Hanoi, Vietnam}
\email{dstiep@math.ac.vn}

\author{Feng Guo$^\ddag$}
\address{School of Mathematical Sciences, Dalian  University of Technology, Dalian, 116024, China}
\email{fguo@dlut.edu.cn}

\author[T. S. Ph\d{a}m]{Ti\'{\^{e}}n-S\OW n Ph\d{a}m$^*$}
\address{Department of Mathematics, Dalat University, 1 Phu Dong Thien Vuong, Dalat, Vietnam}
\email{sonpt@dlu.edu.vn}

\date{\today}


\begin{document}

\begin{abstract}
We present a global version of the {{\L}}ojasiewicz inequality on comparing the rate of growth of two polynomial functions in the case the mapping defined by these functions is (Newton) non-degenerate at infinity. In addition, we show that the condition of non-degeneracy at infinity is generic in the sense that it holds in an open dense semi-algebraic set of the entire space of input data.
\end{abstract}

\maketitle

\section{Introduction}

Let $K$ be a compact semi-algebraic subset of $\mathbb{R}^n$ and let $g, h\colon K\to\mathbb R$ be continuous semi-algebraic functions such that the zero set of $g$ is contained in the zero set of $h.$ Then the information concerning the rate of growth of $g$ and $h$ is given by the following {\em \L ojasiewicz inequality}: there exist constants $c > 0$ and $\alpha > 0$ such that for any $x \in K$, we have
\begin{eqnarray*}
|g(x)|^{\alpha} & \ge & c|h(x)|.
\end{eqnarray*}

Note that if $K$ is not compact, the \L ojasiewicz inequality does not always hold  (see Example~\ref{Example31} below). Recently, several versions of the \L ojasiewicz inequality have been studied for a special case where $h$ is the distance function to the zero set of $g,$ see \cite{Dinh2017-2, Dinh2012, Dinh2014-1, Dinh2013, Dinh2019, Dinh2016-1, HaHV2013, HaHV2015-1}.  However, the study of the {\L}ojasiewicz inequality on comparing the rate of growth of two {\em arbitrary} semi-algebraic functions on {\em non-compact} semi-algebraic sets is barely developed (cf. \cite{Loi2016}).

We would like to point out that the {\L}ojasiewicz inequality and its variants play an important role in many branches of mathematics.
For example, {\L}ojasiewicz inequalities are very useful in the study of continuous regular functions, a branch of Algebraic Geometry, which has been actively developed recently, see \cite{Fichou2016, Kucharz2009} for pioneering works and \cite{Kucharz2019} for a survey. Also, {\L}ojasiewicz inequalities, together with Nullstellens\"atz, are crucial tools for the study of the ring of (bounded) continuous semi-algebraic functions on a semi-algebraic set, see \cite{Fernando2014-1, Fernando2014-2, Fernando2015}. 

The purpose of this work is to show that for almost all pairs of polynomial functions, a variant of the {\L}ojasiewicz inequality holds on the entire space.
Namely, with the definitions given in Section~\ref{SectionPreliminary}, the following statements hold.

\begin{theorem}\label{Theorem11}
Let $(g, h) \colon \mathbb{R}^n \rightarrow \mathbb{R}^2$ be a polynomial mapping, which is non-degenerate at infinity. If $g$ is convenient and $g^{-1}(0) \subset h^{-1}(0),$ then there exist some constants $c > 0, \alpha > 0,$ and $\beta > 0$ such that
\begin{eqnarray*}
|g(x)|^{\alpha} + |g(x)|^{\beta} &\ge& c|h(x)| \quad \textrm{ for all } \quad x \in \mathbb{R}^n.
\end{eqnarray*}
\end{theorem}

\begin{theorem}\label{Theorem12}
In the space of polynomial mappings from $\mathbb{R}^n$ to $\mathbb{R}^p$ $(n \ge p)$ with fixed Newton polyhedra, the set of polynomial mappings, which are non-degenerate at infinity, forms an open dense semi-algebraic subset.
\end{theorem}

Note that unlike the case where $h$ is the distance function to the zero set of $g$ (see \cite{Dinh2017-2,Dinh2014-1,Dinh2016-1,Dinh2017-1,HaHV2015-1}), estimating the exponents $\alpha$ and $\beta$ in Theorem~\ref{Theorem11} is still a delicate problem.

The paper is organized as follows. Section~\ref{SectionPreliminary} presents some preliminary results from Semi-algebraic Geometry; the condition of non-degeneracy at infinity and the Ekeland variational principle will be also given there. Section~\ref{SectionLojasiewicz} proves the existence of the global \L ojasiewicz-type inequality for polynomial mappings which are non-degenerate at infinity.  Finally, in Section \ref{SectionGenericity}, it is shown that the property of being non-degenerate at infinity is generic.

\section{Preliminaries} \label{SectionPreliminary}

We begin by giving some necessary definitions and notational conventions. Let $\mathbb{R}^n$ denote the Euclidean space of dimension $n$  and $\mathbb{R}^*:=\mathbb{R}\backslash\{0\}$. The corresponding inner product (resp., norm) in $\mathbb{R}^n$  is defined by $\langle x, y \rangle$ for any $x, y \in
\mathbb{R}^n$ (resp., $\| x \| := \sqrt{\langle x, x \rangle}$ for any $x \in \mathbb{R}^n$). 
The closure, the convex hull and the cardinality of a set $A$ is denoted by $\overline{A}$, $\co(A)$ and $\# A$ respectively.
\subsection{Semi-algebraic geometry}

In this subsection, we recall some notions and results of semi-algebraic geometry, which can be found in \cite{Basu2006, Benedetti1990, Bierstone1988, Bochnak1998, Dries1996}.

\begin{definition}
\begin{enumerate}[{\rm (i)}]
  \item A subset of $\mathbb{R}^n$ is {\em semi-algebraic} if it is a finite union of sets of the form 
$$\{x \in \mathbb{R}^n \ : \ f_i(x) = 0, i = 1, \ldots, k; f_i(x) > 0, i = k + 1, \ldots, p\}$$
where all $f_{i}$ are polynomials.
 \item 
%Let $A \subset \mathbb{R}^n$ and $B \subset \mathbb{R}^p$ be semi-algebraic sets. 
	 A mapping $F \colon A \to B$ is {\em semi-algebraic} if its graph 
$$\{(x, y) \in A \times B \ : \ y = F(x)\}$$
is a semi-algebraic subset of $\mathbb{R}^n\times\mathbb{R}^p.$
\end{enumerate}
\end{definition}

A major fact concerning the class of semi-algebraic sets is the following Tarski--Seidenberg Theorem.

\begin{theorem} \label{TarskiSeidenbergTheorem1}
The image of a semi-algebraic set by a semi-algebraic mapping is semi-algebraic.
\end{theorem}

% Moreover, semi-algebraic sets and mappings enjoy a number of remarkable properties:
% \begin{enumerate}[{\rm (i)}]
% \item The class of semi-algebraic sets is closed with respect to Boolean operators; a Cartesian product of semi-algebraic sets is a semi-algebraic set;
% \item The closure and the interior of a semi-algebraic set is a semi-algebraic set;
% \item A composition of semi-algebraic mappings is a semi-algebraic mapping;
% \item The  inverse image of  a  semi-algebraic set  under  a  semi-algebraic mapping is a semi-algebraic set;
% \item If $A$ is a semi-algebraic set, then the function distance to
% 	$A$
% $$\mathrm{dist}(\cdot, A) \colon {\mathbb R}^n \rightarrow {\mathbb R}, \quad x \mapsto \mathrm{dist}(x, A) := \inf \{\|x - a\| \ : \ a \in A \},$$
% is also semi-algebraic.
% \end{enumerate}

The following well-known lemmas will be of great importance for us.

\begin{lemma}[Curve Selection Lemma] \label{CurveSelectionLemma1}
Let $A\subset \mathbb{R}^n$ be a semi-algebraic set and $x^* \in \mathbb{R}^n$ be a non-isolated point of $\overline{A}.$ 
Then there exists a non-constant analytic semi-algebraic mapping  $\varphi\colon (-\epsilon, \epsilon) \to {\mathbb R}^n$ with $\varphi(0) = x^*$ and with $\varphi(t) \in A$ for $t \in (0, \epsilon).$
\end{lemma}

\begin{lemma}[Curve Selection Lemma at infinity]\label{CurveSelectionLemma}
Let $A\subset \mathbb{R}^n$ be a semi-algebraic set, and let $f :=
(f_1, \ldots,f_p) \colon  \mathbb{R}^n \to \mathbb{R}^p$ be a
semi-algebraic mapping. Assume that there exists a sequence $\{x^\ell\}$
such that $x^\ell \in A$, $\lim_{l \to \infty} \| x^\ell  \| = \infty$
and $\lim_{l \to \infty} f(x^\ell)  = y \in(\overline{\mathbb{R}})^p,$
where $\overline{\mathbb{R}} := \mathbb{R} \cup \{\pm \infty\}.$ Then
there exists an analytic semi-algebraic mapping $\varphi \colon (0,
\epsilon)\to \mathbb{R}^n$ such that $\varphi(t) \in A$ for all $t \in
(0, \epsilon), \lim_{t \to 0} \|\varphi(t)\| = \infty,$ and $\lim_{t
\to 0} f(\varphi(t)) = y.$
\end{lemma}

\begin{lemma}[Growth Dichotomy Lemma] \label{GrowthDichotomyLemma}
Let $f \colon (0, \epsilon) \rightarrow \mathbb{R}$ be a semi-algebraic function with $f(t) \ne 0$ for all $t \in (0, \epsilon).$ Then
there exist constants $c \ne 0$ and $q \in \mathbb{Q}$ such that $f(t) = ct^q + o(t^q)$ as $t \to 0^+.$
\end{lemma}

\subsection{The (semi-algebraic) transversality theorem with parameters}
Let $P,X$ and $Y$ be $C^\infty$ manifolds of finite dimension, $S$ be a $C^\infty$ sub-manifold of $Y$, and $F \colon X\rightarrow Y$ be a $C^\infty$ mapping. Denote by $d_xF:T_xX\to T_{F(x)}Y$, the differential of $F$ at $x,$ where $T_xX$ and $T_{F(x)}Y$ are, respectively, the tangent space of $X$ at $x$ and the tangent space of $Y$ at $F(x)$. 

\begin{definition}
The mapping $F$ is {\em transverse} to the sub-manifold $S,$ abbreviated by $F \pitchfork S,$ if either $F(X)\cap S=\emptyset$ or for each $x\in F^{-1}(S),$ we have
$$d_xF(T_x X)+T_{F(x)}S=T_{F(x)}Y.$$
\end{definition}

\begin{remark}
If $\dim X \ge \dim Y$ and $S=\{s\}$, then $F\pitchfork S$ if and only if either $F^{-1}(s)=\emptyset$ or $\mathrm{rank} d_xF = \dim Y$ for all $x \in F^{-1}(s).$ In the case $\dim X < \dim Y$, then $F\pitchfork S$ if and only if $F^{-1}(S)=\emptyset$.
\end{remark}

The following result will be useful in the study of the genericity of the condition of non-degeneracy at infinity.
\begin{theorem} [Transversality Theorem with parameters] \label{SardTheorem}
Let $F \colon P\times X\rightarrow Y$ be a $C^\infty$ semi-algebraic mapping. For each $p \in P,$ consider the mapping $F_p \colon X\rightarrow Y$ defined by $F_p(x) := F(p, x).$ If $F\pitchfork S,$ then the set
$$Q := \{p\in P \ : \  F_p\pitchfork S\}$$
contains an open dense semi-algebraic subset of $P.$
\end{theorem}
\begin{proof} 
It is well-known that the set $Q$ is dense in $P$ (see, for example, \cite[Theorem~1.3.6]{Goresky1988} or \cite[The Transversality Theorem, page 68]{Guillemin1974}).
On the other hand, $Q$ is semi-algebraic (by Theorem~\ref{TarskiSeidenbergTheorem1}). Since every dense semi-algebraic set in $P$ contains an open dense semi-algebraic subset of $P,$ the desired statement follows.
\end{proof}

\subsection{Newton polyhedra and non-degeneracy conditions}
Given a nonempty set $J \subset \{1, \ldots, n\},$ we define
$$\mathbb{R}^J := \{x \in \mathbb{R}^n \ : \ x_j = 0, \textrm{ for all } j \not \in J\}.$$
We denote by $\mathbb{Z}_+$ and $\mathbb{R}_+$, respectively, the set of non-negative integers and the set of non-negative real numbers. If $\kappa = (\kappa_1, \ldots, \kappa_n) \in \mathbb{Z}_+^n,$ we denote by $x^\kappa$ the monomial $x_1^{\kappa_1} \cdots x_n^{\kappa_n}.$
Denote by $\{e^1, \ldots, e^n\}$ the canonical basis of $\mathbb{R}^n.$

\subsubsection{Newton polyhedra}
A subset $\Gamma \subset {\mathbb R}^n_+$ is a {\em Newton polyhedron} if there exists a finite subset $S \subset {\mathbb Z}^n_+$  such that $\Gamma$ is the convex hull in ${\mathbb R}^n$ of $S.$ We say that $\Gamma$ is the {\em Newton polyhedron determined by $S$} and write $\Gamma = \Gamma(S).$ 
A Newton polyhedron $\Gamma$ is {\em convenient} if it intersects each
coordinate axis at a point different from the origin $0$ in
$\mathbb{R}^n,$ that is, if for any $j \in \{1, \ldots, n\}$ there
exists some $\kappa_j > 0$ such that $\kappa_j e^j \in \Gamma.$

Given a Newton polyhedron $\Gamma$ and a vector $q \in {\mathbb R}^n,$ we define
\begin{equation*}
\begin{aligned}
d(q, \Gamma) :=& \min \{\langle q, \kappa \rangle \ : \ \kappa \in \Gamma\}, \\
\Delta(q, \Gamma) :=& \{\kappa \in \Gamma \ : \ \langle q, \kappa \rangle = d(q, \Gamma) \}.
\end{aligned}
\end{equation*}
%We say that a subset $\Delta$ of $\Gamma$ is a {\em face} of $\Gamma$ if there exists a vector $q \in {\mathbb R}^n$ such that $\Delta = \Delta(q, \Gamma).$ The dimension of a face $\Delta$ is defined as the minimum of the dimensions of the affine subspaces containing $\Delta.$ The faces of $\Gamma$ of dimension $0$ are the {\em vertices} of $\Gamma.$ 
By definition, for each nonzero vector $q \in \mathbb{R}^n,$
$\Delta(q, \Gamma) $ is a closed face of $\Gamma.$ Conversely,
if $\Delta$ is a closed face of $\Gamma$, then there exists a
nonzero vector $q\in \mathbb{R}^n$ such that $\Delta = \Delta(q, \Gamma),$ where we
can in fact assume that $q\in\mathbb{Q}^n$ since $\Gamma$ is an integer polyhedron. 
The {\em dimension} of a face $\Delta$ is the minimum of the
dimensions of the affine subspaces containing $\Delta.$ 
The faces of $\Gamma$ of dimension $0$ are the {\em vertices} of $\Gamma.$

Let $f \colon \mathbb{R}^n \to \mathbb{R}$ be a polynomial function. Suppose that $f$ is written as $f = \sum_{\kappa} c_\kappa x^\kappa.$ The {\em support} of $f,$ denoted by $\mathrm{supp}(f),$ is the set of $\kappa  \in \mathbb{Z}_+^n$ such that $c_\kappa \ne 0.$ 
The {\em Newton polyhedron (at infinity)} of $f$, denoted by $\Gamma(f),$ is the convex hull in $\mathbb{R}^n$ of the set $\mathrm{supp}(f),$ i.e.,  $\Gamma(f)=\Gamma(\mathrm{supp}(f)).$
The polynomial $f$ is {\em convenient} if $\Gamma(f)$ is convenient.
For each (closed) face $\Delta$ of $\Gamma(f),$  we will denote 
$$f_\Delta(x):=\sum_{\kappa \in \Delta} c_\kappa x^\kappa.$$

\begin{remark}
The following statements follow immediately from definitions:

(i) We have $\Gamma(f) \cap \mathbb{R}^J = \Gamma(f|_{\mathbb{R}^J})$  for all nonempty subset $J$ of $\{1, \ldots, n\}.$

(ii) Let $\Delta := \Delta(q, \Gamma(f))$ for some nonzero vector $q
:= (q_1, \ldots, q_n) \in \mathbb{R}^n.$ By definition,
$f_\Delta(x)$ is a weighted homogeneous polynomial of type $(q, d := d(q, \Gamma(f))),$ i.e., we have for all $t > 0$ and all $x \in \mathbb{R}^n,$
$$f_\Delta(t^{q_1} x_1,  \ldots, t^{q_n} x_n) = t^d f_\Delta(x_1, \ldots, x_n).$$
This implies the Euler relation
\begin{equation*}
\sum_{j = 1}^n q_j x_j \frac{\partial f_\Delta}{\partial x_j}(x) = d \cdot f_\Delta(x).
\end{equation*}
In particular, if $d \ne 0$ and $\nabla f_\Delta(x) = 0,$ then $f_\Delta(x) = 0.$
\end{remark}

\subsubsection{Non-degeneracy conditions}
In \cite{Khovanskii1978} (see also \cite{Kouchnirenko1976}),
Khovanskii introduced a condition of non-degeneracy for complex
analytic mappings $F \colon ({\mathbb C}^n, 0) \rightarrow ({\mathbb C}^p,
0)$ in terms of the Newton polyhedra of the component functions of
$F.$ This notion has been applied extensively to the study of isolated
complete intersection singularities (see for instance \cite{Bivia2007,
Damon1989, Gaffney1992, Oka1997}). We will use this condition for real
polynomial mappings. First we need to introduce some notation.

\begin{definition}
Let $F := (f_1, \ldots, f_p) \colon {\mathbb R}^n \rightarrow {\mathbb R}^p, 1 \le p \le n,$ be a polynomial mapping. 
\begin{enumerate}[{\rm (i)}]
\item The mapping $F$ is {\em Khovanskii non-degenerate at infinity} if for any vector $q \in \mathbb{R}^n$ with $d(q, \Gamma(f_i)) < 0$ for $i = 1, \ldots, p,$ 
the system of gradient vectors $\nabla f_{i, \Delta_i}(x)$, for $i = 1, \ldots, p$, is $\mathbb{R}$-linearly independent on the set
	\begin{eqnarray*}
\{x \in (\mathbb{R}^*)^{n} \ : \ f_{i, \Delta_i} (x) = 0 \textrm{ for } i = 1, \ldots, p\},
\end{eqnarray*}
where $\Delta_i := \Delta(q, \Gamma(f_i)).$

% is a reduced smooth complete intersection variety in the torus
% $(\mathbb{R}^*)^{n}.$ 
% By definition, $F$ is {\em Khovanskii non-degenerate at infinity} if
% the system of gradient vectors $\nabla f_{i, \Delta_i}(x)$ for $i = 1,
% \ldots, p,$ is $\mathbb{R}$-linearly independent on the variety defined by \eqref{eq::v}.

\item The mapping $F$ is {\em non-degenerate at infinity} if for each $k$-tuple $(i_1, \ldots, i_k)$ of integers with $1 \le i_1 < \cdots < i_k \le p,$ the polynomial mapping 
	$${\mathbb R}^n \rightarrow {\mathbb R}^k,\ x \mapsto (f_{i_1}(x), \ldots,	f_{i_k}(x)),$$ 
	is Khovanskii non-degenerate at infinity.
\end{enumerate}
\end{definition}

\begin{remark}
By definition, the mapping $F$ is Khovanskii non-degenerate at infinity if and only
if for all $q \in \mathbb{R}^n$ with $d(q, \Gamma(f_i)) < 0$ for $i = 1, \ldots, p$ and for all $x \in ({\mathbb R}^*)^n$ with $f_{i, \Delta_i} (x) = 0$ for $i = 1, \ldots, p$ we have
$$\mathrm{rank} 
\begin{pmatrix}
x_1\frac{\partial f_{1, \Delta_1}}{\partial x_1}(x) & \cdots & x_n\frac{\partial f_{1, \Delta_1}}{\partial x_n}(x) \\
\vdots & \cdots & \vdots \\
x_1\frac{\partial f_{p, \Delta_p}}{\partial x_1}(x) & \cdots & x_n\frac{\partial f_{p, \Delta_p}}{\partial x_n}(x)
\end{pmatrix} = p.$$
The mapping $F$ is non-degenerate at infinity if and only
if for all $q \in \mathbb{R}^n$ with $d(q, \Gamma(f_i)) < 0$ for $i =
1, \ldots, p,$ and for all $x \in ({\mathbb R}^*)^n,$ we have
$$\mathrm{rank} 
\begin{pmatrix}
x_1\frac{\partial f_{1, \Delta_1}}{\partial x_1}(x) & \cdots &
x_n\frac{\partial f_{1, \Delta_1}}{\partial x_n}(x) & f_{1,\Delta_1}(x) &   & \textbf{\LARGE 0} \\
\vdots & \cdots & \vdots & & \ddots & \\
x_1\frac{\partial f_{p, \Delta_p}}{\partial x_1}(x) & \cdots &
x_n\frac{\partial f_{p, \Delta_p}}{\partial x_n}(x) & \textbf{\LARGE
0} &  & f_{p,\Delta_p}(x)
\end{pmatrix} = p.$$
\end{remark}

\subsection{Ekeland's variational principle}
We recall the Ekeland variational principle which is important for our arguments in the following.

\begin{theorem}\cite[Ekeland's Variational Principle]{Ekeland1974}\label{th::ekeland}
Let $X \subset \mathbb{R}^n$ be a closed set, and $f \colon X \rightarrow \mathbb{R}$ be a continuous function$,$ bounded from below. Let $\epsilon>0$ and $x^{0} \in X$ be such that
\[
		\inf_{x \in X} f(x)\le f(x^{0}) \le \inf_{x \in X} f(x)+\epsilon.
\]
Then for any $\lambda>0,$ there exists some point $y^{0} \in X$
such that 
\begin{eqnarray*}
			&& f(y^{0})\le f(x^{0}),\\
			&& \Vert y^{0}-x^{0}\Vert\le\lambda,\\
			&& f(y^{0}) \le f(x) + \frac{\epsilon}{\lambda}\Vert x-y^{0}\Vert \quad \text{ for all }\ x \in X.
\end{eqnarray*}
\end{theorem}

\section{\L ojasiewicz inequalities}\label{SectionLojasiewicz}

The main purpose of this section is to prove Theorem \ref{Theorem11}, which gives a global {\L}ojasiewicz inequality on comparing the rate of
growth of two polynomial functions. 
Note that we do not suppose the polynomial $h$ to be convenient. On the other hand, the assumption that the polynomial $g$ is convenient cannot be dropped. This is shown in the following example.

\begin{example}\label{Example31}
Consider the polynomial mapping
\begin{eqnarray*}
(g, h) \colon \mathbb{R}^2 \rightarrow \mathbb{R}^2, \quad (x_1, x_2) \mapsto \big((x_1^2 - 1)^2 + (x_1x_2 - 1)^2, (x_1^2 - 1)^2 + (x_2^2 - 1)^2\big).
\end{eqnarray*}
Clearly, $(g, h)$ is non-degenerate at infinity, $g$ is not convenient
(see Figure~1), and 
$$g^{-1}(0) \subset h^{-1}(0).$$ 

\begin{figure}[h]
\unitlength = 1.25cm
$$\begin{picture}(5, 5)(1, -.25)
\put (0, 0){\vector(0, 1){4.5}}
\put (0, 0){\vector(1, 0){5}}

{\color{blue}
\put (0, 0){\line(1, 1){2}}
\put (0, 0){\line(1, 0){4}}
\put (2, 2){\line(1, -1){2}}
}
\put (5.15, 0){$x_1$}
\put (0, 4.75){$x_2$}

\put (1, -.15){\line(0, 1){.15}}
\put (2, -.15){\line(0, 1){.15}}
\put (3, -.15){\line(0, 1){.15}}
\put (4, -.15){\line(0, 1){.15}}

\put (-.15, 1){\line(1, 0){.15}}
\put (-.15, 2){\line(1, 0){.15}}
\put (-.15, 3){\line(1, 0){.15}}
\put (-.15, 4){\line(1, 0){.15}}

{\color{red}
\put (0, 0){\circle*{.15}}
\put (2, 0){\circle*{.15}}
\put (1, 1){\circle*{.15}}
\put (2, 2){\circle*{.15}}
\put (4, 0){\circle*{.15}}}

\put (0, -.5){$0$}
\put (1, -.5){$1$}
\put (2, -.5){$2$}
\put (3, -.5){$3$}
\put (4, -.5){$4$}
\put (-.35, 1){$1$}
\put (-.35, 2){$2$}
\put (-.35, 3){$3$}
\put (-.35, 4){$4$}
\put (2, 2.5){$\Gamma(g)$}
\end{picture}
\begin{picture}(5, 4)(-1., -.25)
\put (0, 0){\vector(0, 1){4.5}}
\put (0, 0){\vector(1, 0){5}}

{\color{blue}
\put (0, 4){\line(1, -1){4}}
\put (0, 0){\line(1, 0){4}}
\put (0, 0){\line(0, 1){4}}
}

\put (5.15, 0){$x_1$}
\put (0, 4.75){$x_2$}

\put (1, -.15){\line(0, 1){.15}}
\put (2, -.15){\line(0, 1){.15}}
\put (3, -.15){\line(0, 1){.15}}
\put (4, -.15){\line(0, 1){.15}}

\put (-.15, 1){\line(1, 0){.15}}
\put (-.15, 2){\line(1, 0){.15}}
\put (-.15, 3){\line(1, 0){.15}}
\put (-.15, 4){\line(1, 0){.15}}

{\color{red}
\put (0, 0){\circle*{.15}}
\put (2, 0){\circle*{.15}}
\put (4, 0){\circle*{.15}}
\put (0, 2){\circle*{.15}}
\put (0, 4){\circle*{.15}}}

\put (0, -.5){$0$}
\put (1, -.5){$1$}
\put (2, -.5){$2$}
\put (3, -.5){$3$}
\put (4, -.5){$4$}

\put (-.35, 1){$1$}
\put (-.35, 2){$2$}
\put (-.35, 3){$3$}
\put (-.35, 4){$4$}
\put (2.25, 2){$\Gamma(h)$}
\end{picture}$$
\caption{The Newton polyhedra of $g$ and $h$ in Example~\ref{Example31}.}
\end{figure}
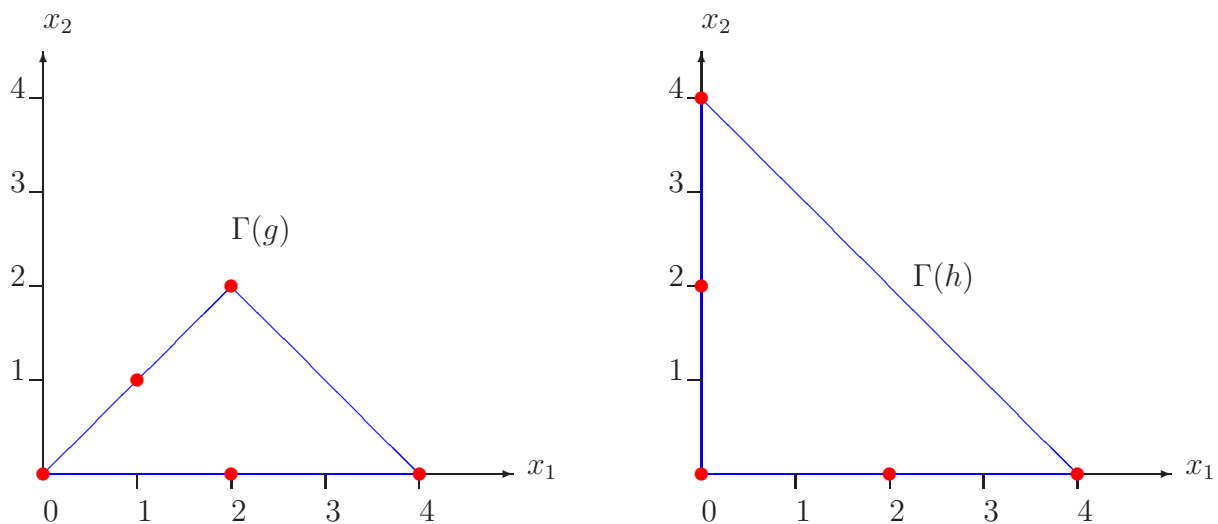

Furthermore, we have
\begin{eqnarray*}
\lim_{k \to \infty} g\big(\frac{1}{k}, k\big) &=& 1 \quad \textrm{ and } \quad \lim_{k \to \infty} h\big(\frac{1}{k}, k\big) \ = \ +\infty,
\end{eqnarray*}
and so there are no constants $c > 0, \alpha > 0,$ and $\beta > 0$ such that
\begin{eqnarray*}
|g(x)|^{\alpha} + |g(x)|^{\beta} &\ge& c|h(x)| \quad \textrm{ for all } \quad x = (x_1, x_2) \in \mathbb{R}^2.
\end{eqnarray*}
\end{example}

The following simple example shows that the exponents $\alpha$ and $\beta$ in Theorem~\ref{Theorem11} are different in general.
\begin{example}
Consider the polynomial mapping
\begin{eqnarray*}
(g, h) \colon \mathbb{R}^2 \rightarrow \mathbb{R}^2, \quad (x_1, x_2) \mapsto \big(x_1^2 + x_2^4, x_1^2 + x_2^2\big).
\end{eqnarray*}
Clearly, $(g, h)$ is non-degenerate at infinity, $g$ is convenient, and $g^{-1}(0) \subset h^{-1}(0).$ 
Furthermore, it is not hard to see that there are no constants $c > 0$ and $\alpha > 0$ such that
\begin{eqnarray*}
|g(x)|^{\alpha} &\ge& c|h(x)| \quad \textrm{ for all } \quad x = (x_1, x_2) \in \mathbb{R}^2.
\end{eqnarray*}
On the other hand, it holds that
\begin{eqnarray*}
|g(x)|^{\frac{1}{2}} + |g(x)| &\ge& |h(x)| \quad \textrm{ for all } \quad x = (x_1, x_2) \in \mathbb{R}^2.
\end{eqnarray*}

\end{example}

To prove Theorem~\ref{Theorem11}, we first need the following definition.

\begin{definition}
Given a polynomial function $f \colon \mathbb{R}^n \rightarrow \mathbb{R}$ and a smooth semi-algebraic manifold $X \subset \mathbb{R}^n$ we let
$$\widetilde{K}_\infty(f|_X) := \left\{
t \in \mathbb{R}: \ \left\{ 
	\begin{aligned}
		&\exists \{x^k\} \subset X, \textrm{ s.t. } \|x^k\| \to +\infty, f(x^k) \to t, \\  
		&\textrm{ and } \|\nabla (f|_X)(x^k)\| \to 0
	\end{aligned}
\right.\right\}.$$
We also set 
\begin{eqnarray*}
{K}_0(f|_X) &:=& \{t \in \mathbb{R} \ : \ \exists x \in X \textrm{ with } f(x) = t  \textrm{ and } \nabla (f|_X) (x) = 0 \}
\end{eqnarray*}
which is the {\em set of critical values} of the restriction $f|_X.$
Note that, by the semi-algebraic Sard Theorem (see \cite[Theorem~1.9]{HaHV2017} and \cite{Kurdyka2000-1}), ${K}_0(f|_X)$ is a finite subset of $\mathbb{R}.$ 

If $X = \mathbb{R}^n$, we write $\widetilde{K}_\infty(f) $ and ${K}_0(f) $ instead of $\widetilde{K}_\infty(f|_{\mathbb{R}^n})$ and ${K}_0(f|_{\mathbb{R}^n}),$ respectively.
\end{definition}

\begin{lemma} \label{Lemma31-1}
Let $g \colon \mathbb{R}^n \rightarrow \mathbb{R}$ be a polynomial function, which is non-degenerate at infinity and convenient, then
$\widetilde{K}_\infty(g) = \emptyset.$
\end{lemma}

\begin{proof}
%For simplicity of notation, we write $g$ and $f_2$ instead of $g$ and $h,$ respectively. 
%	Note that the proof for $\widetilde{K}_\infty(g) = \emptyset$ can be found in \cite[Theorem 1]{Dinh2014-1}; since we will use some facts from the proof (and for the sake of completeness), we include the proof of this statement here. 
The result can be found in \cite[Theorem 1]{Dinh2014-1}. For the sake of completeness, we include the proof below.
By contradiction, suppose that there exist a sequence $\{x^k\}_{k \in {\mathbb N}} \subset {\mathbb R}^n$ and a value $y \in {\mathbb R}$ such that
$$\lim_{k \to \infty} \|x^k \|= \infty, \quad \lim_{k \to \infty}  g(x^k) = y,  \quad \textrm{ and } \quad \lim_{k \to \infty} \|\nabla g(x^k)\| = 0.$$
By Lemma~\ref{CurveSelectionLemma}, there exists an analytic mapping
$\varphi \colon (0, \epsilon) \rightarrow \mathbb{R}^n, t \mapsto
(\varphi_1(t), \ldots, \varphi_n(t)),$ such that
\begin{enumerate}
  \item [(a1)] $\lim_{t \to 0} \| \varphi(t)\| = \infty;$
  \item [(a2)] $\lim_{t \to 0} g(\varphi(t)) = y;$ and
  \item [(a3)] $\lim_{t \to 0}  \|\nabla g(\varphi(t))\| = 0.$
\end{enumerate}

Let $J := \{j \ : \ \varphi_j \not\equiv 0\}.$ By Condition (a1), 
$J \ne \emptyset.$ By Lemma~\ref{GrowthDichotomyLemma}, for each $j \in J,$ we expand the coordinate functions $\varphi_j$ as follows
\begin{eqnarray*}
	\varphi_j(t) &=&  x_j^0 t^{q_j} + o(t^{q_j}),
	%\textrm{ higher order terms in } t,
\end{eqnarray*}
where $x_j^0 \ne 0$ and $q_j \in \mathbb{Q}.$ From Condition~(a1), we get $\min_{j \in J} q_j  < 0.$

Let $q := (q_1, \ldots, q_n) \in \mathbb{R}^n,$ where $q_j := M$ for $j \not \in J$ with $M$ being sufficiently large and satisfying
\begin{eqnarray*}
M & > & \max \left  \{\sum_{j \in J} q_j \kappa_j \ : \ \kappa \in \Gamma(g) \right\}.
\end{eqnarray*}
Let $d$ be the minimal value of the linear function $\sum_{j = 1}^n q_j \kappa_j$ on $\Gamma(g)$ and let $\Delta$ be the maximal face of $\Gamma(g)$ (maximal with respect to the inclusion of faces)  where the linear function takes this value, i.e.,
\begin{eqnarray*}
d &:=& d(q, \Gamma(g)) \quad \textrm{ and } \quad \Delta \ := \ \Delta(q, \Gamma(g)).
\end{eqnarray*}

Recall that ${\mathbb R}^J := \{x := (x_1, x_2, \ldots, x_n) \in
	{\mathbb R}^n \ : \ x_j = 0 \textrm{ for } j \not
\in J \}.$ Since $g$ is convenient, the restriction
$g|_{\mathbb{R}^J}$ is not constant, $\Gamma(g) \cap {\mathbb{R}^J} =
\Gamma(g|_{\mathbb{R}^J}) $ is nonempty and different from $\{0\}.$
Furthermore,  by definition of the vector $q,$ one has 
\begin{eqnarray*}
d &=& d(q, \Gamma(g|_{\mathbb{R}^J})) \quad \textrm{ and } \quad \Delta \ = \ \Delta(q, \Gamma(g|_{\mathbb{R}^J})) 
\ \subset \ {\mathbb{R}^J}.
\end{eqnarray*}
%In particular, for each $j \not \in J,$ the polynomial $g_{\Delta}$ does not depend on the variable $x_j.$ 
A direct calculation shows that
\begin{eqnarray*}
	g(\varphi(t)) &=& g_{\Delta}(x^0)t^{d} + o(t^{d}),
%\textrm{ higher order terms in } t,
\end{eqnarray*}
where $x^0 := (x_1^0, \ldots, x_n^0)$ with $x_j^0: = 1$ for $j \not \in J$
(note that, for $j \not \in J$, as $g_{\Delta}$ does not depend on the variable $x_j$, we can choose $x_j^0$ arbitrarily).
%It is easy to see that $d\le q_{j_*}  := \min_{j \in J} q_j.$ In fact, 
Since $g$ is convenient, for each $j = 1, \ldots, n,$ there exists a
natural number $m_j \ge 1$ such that $m_j e^j \in \Gamma(g).$ Let $j_* \in J$ be such that $q_{j_*} := \min_{j \in J} q_j$. As $q_{j_*}<0$, it is clear that
$$d\le q_{j_*} m_{j_*}\le q_{j_*} <0.$$
Now, by Condition (a2), we have $g_{\Delta}(x^0) = 0$.

On the other hand, for $j \in J$, we have
\begin{eqnarray*}
\frac{\partial g}{\partial x_j}(\varphi(t)) &=& \frac{\partial
	g_{\Delta}}{\partial x_j}(x^0)t^{d - q_j}  + o(t^{d - q_j}).
%\textrm{ higher order terms in } t.
\end{eqnarray*}
Since $d\le \min_{j\in J} q_j,$ it follows from (a3) that $\frac{\partial g_{\Delta}}{\partial x_j}(x^0)= 0$ for all $j \in J.$
So this, together with $g_{\Delta}(x^0) = 0,$ implies that $g$ is not
Khovanskii non-degenerate at infinity, which contradicts our
assumption.
\end{proof}

\begin{lemma} \label{Lemma31-2}
Let $(g, h) \colon \mathbb{R}^n \rightarrow \mathbb{R}^2$ be a
polynomial mapping, which is non-degenerate at infinity. If $g$ is
convenient, then $\widetilde{K}_\infty(g|_{\{h = r\}}) = \emptyset$
for $0 < |r| \ll 1$ and for $|r| \gg 1.$ 
\end{lemma}
\begin{proof}
For simplicity of notation, we write $f_1$ and $f_2$ instead of $g$ and $h,$ respectively.
We will show that $\widetilde{K}_\infty(f_1|_{\{f_2 = r\}}) = \emptyset$ 
for $0 < |r| \ll 1$ and for $|r| \gg 1.$ 

For $0 < |r| \ll 1$ or $|r| \gg 1,$ in view of the semi-algebraic Sard Theorem, we can make the following assumptions without loss of generality: 
\begin{enumerate}
\item [{(i)}] $r\not\in K_0(f_{2,\Delta_2})$ for any face $\Delta_2$ of $\Gamma(f_2)$;
\item [{(ii)}] If the set $X := \{x \in (\mathbb{R}^*)^n \ : \  f_{1,\Delta_1} (x) = 0, \nabla f_{1,\Delta_1} (x)  \ne 0\}$ is not empty for some face 
$\Delta_1$ of $\Gamma(f_1)$, then $r \not\in K_0(f_{2,\Delta_2}|_X)$ for any face $\Delta_2$ of $\Gamma(f_2).$ (Clearly, if $X$ is not empty, then it a semi-algebraic smooth hypersurface.)
\end{enumerate}

By contradiction, suppose that $\widetilde{K}_\infty(f_1|_{\{f_2=r\}})\not = \emptyset$ for some  $0 < |r| \ll 1$ or $|r| \gg 1,$ i.e., 
there exist a sequence $\{x^k\}_{k \in {\mathbb N}} \subset {\mathbb R}^n$ and a value $y \in {\mathbb R}$ such that
$$\lim_{k \to \infty} \|x^k \|= \infty, \quad \lim_{k \to \infty}  f_1(x^k) = y, \quad f_2(x^k) = r,  \quad \textrm{ and } \quad
\lim_{k \to \infty} \|\nabla (f_1|_{\{f_2 = r\}})(x^k)\| = 0.$$
By definition, there exists a sequence $\lambda^k \in {\mathbb R}$ such
that for all $k \ge 1,$ we have 
\begin{eqnarray*}
	&& \nabla (f_1|_{\{f_2 = r\}})(x^k)\ = \ \nabla f_1(x^k)-\lambda^k\nabla
	f_2(x^k).
\end{eqnarray*}
By Lemma~\ref{CurveSelectionLemma}, there exist an analytic mapping
$\varphi(t) := (\varphi_1(t), \ldots, \varphi_n(t))$ and an analytic function $\lambda(t)$ with $0 < t \ll 1$ such that
\begin{enumerate}
  \item [(b1)] $\lim_{t \to 0} \| \varphi(t)\| = \infty;$
  \item [(b2)] $\lim_{t \to 0} f_1(\varphi(t)) = y;$
  \item [(b3)] $f_2(\varphi(t)) = r;$ and
  \item [(b4)] $\lim_{t \to 0}  \|\nabla f_1(\varphi(t))-\lambda(t)\nabla f_2(\varphi(t))\| = 0.$
\end{enumerate}
Let $J := \{j \ : \ \varphi_j \not\equiv 0\}\ne \emptyset$ and for each $j \in J,$ expand $\varphi_j$ as follows
\begin{eqnarray*}
\varphi_j(t) &=&  x_j^0 t^{q_j} + o(t^{q_j}),
%\textrm{ higher order terms in } t,
\end{eqnarray*}
where $x_j^0 \ne 0$ and $q_j  \in \mathbb{Q}.$ Let $q := (q_1, \ldots, q_n) \in \mathbb{R}^n,$ where $q_j := M$ for $j \not \in J$ with $M$ being sufficiently large and satisfying
\begin{eqnarray*}
M & > & \max_{i = 1, 2}\left  \{\sum_{j \in J} q_j \kappa_j \ : \ \kappa \in \Gamma(f_i) \right\}.
\end{eqnarray*}
For each $i = 1, 2,$ let $d_i$ be the minimal value of the linear function $\sum_{j = 1}^n q_j \kappa_j$ on $\Gamma(f_i)$ and let $\Delta_i$ be the maximal face of $\Gamma(f_i)$ (maximal with respect to the inclusion of faces)  where the linear function takes this value, i.e.,
\begin{eqnarray*}
d_i &:=& d(q, \Gamma(f_i)) \quad \textrm{ and } \quad \Delta_i \ := \ \Delta(q, \Gamma(f_i)).
\end{eqnarray*}

Since $f_1$ is convenient, the restriction $f_1|_{\mathbb{R}^J}$ is not constant. Furthermore, the restriction $f_2|_{\mathbb{R}^J}$ is not constant. If this is not the case, then it follows from (b4) that
\begin{eqnarray*}
\lim_{t\rightarrow 0}\frac{\partial f_{1}}{\partial x_j}(\varphi(t)) &=& 0 \quad \textrm {for all } \quad j \in J.
\end{eqnarray*}
Replacing $f_1$ by the restriction $f_1|_{\mathbb{R}^J}$ and repeating
the previous arguments, we see that $f_1$ is not Khovanskii
non-degenerate at infinity. Then, the polynomial mapping $(f_1, f_2)$ is not non-degenerate at infinity, which contradicts our assumption.

Therefore, the restriction of $f_i, i = 1, 2,$ on $\mathbb{R}^J$ is not constant, and so $\Gamma(f_i) \cap {\mathbb{R}^J} = \Gamma(f_i|_{\mathbb{R}^J}) $ is nonempty and different from $\{0\}.$ Furthermore,  by definition of the vector $q,$ one has 
\begin{eqnarray*}
d_i &=& d(q, \Gamma(f_i|_{\mathbb{R}^J})) \quad \textrm{ and } \quad \Delta_i \ = \ \Delta(q, \Gamma(f_i|_{\mathbb{R}^J})) 
\ \subset \ {\mathbb{R}^J}.
\end{eqnarray*}
%In particular, for each $j \not \in J,$ the polynomial $f_{i,\Delta_i}$ does not depend on the variable $x_j.$

Similarly to what was done in the proof of Lemma~\ref{Lemma31-1}, we obtain $d_1 \le q_{j_*}  := \min_{j \in J} q_j < 0$ and $f_{1,\Delta_1}(x^0) = 0,$ where $x^0 := (x_1^0, \ldots, x_n^0)$ with $x_j^0: = 1$ for $j \not \in J.$ 
As  $f_{1,\Delta_1}$ and $f_{2,\Delta_2}$ do not depend on the variable $x_j$ for all $j \not \in J$, we have
\begin{eqnarray*}
\frac{\partial f_{1, \Delta_1}}{\partial x_j}(x^0) &=& 
\frac{\partial f_{2, \Delta_2}}{\partial x_j}(x^0)  \ = \ 0\ \text{ for }\ j \not \in J.
\end{eqnarray*}

Note that $\lambda(t) \not\equiv 0,$ since otherwise $y \in \widetilde{K}_\infty(f_1) = \emptyset,$ a contradiction. Hence, we can expand the coordinate $\lambda(t)$ in terms of $t$ as
\begin{eqnarray*}
\lambda(t) &=&  \lambda^0 t^{\theta} + o(t^{\theta}),
%\textrm{ higher order terms in } t,
\end{eqnarray*}
where $\lambda^0 \neq 0$ and $\theta \in \mathbb{Q}.$ There are three cases to be considered.

\subsubsection*{Case 1: $d_1<d_2+\theta$.} For each $j\in J$, we have 
\begin{eqnarray*}
\frac{\partial f_1}{\partial x_j}(\varphi(t))-\lambda(t) \frac{\partial f_2}{\partial x_j}(\varphi(t))
&=& \frac{\partial f_{1, \Delta_1}}{\partial x_j}(x^0)t^{d_1-q_j}  +
o(t^{d_1-q_j}).
%\textrm{ higher order terms in } t. 
\end{eqnarray*}
Since $d_1\le q_{j_*}$, in view of (b4), we have $\frac{\partial f_{1,
\Delta_1}}{\partial x_j}(x^0) = 0$ for all $j \in J.$ As
$f_{1,\Delta_1}(x^0) = 0,$ it implies that $f_1$ is not Khovanskii
non-degenerate at infinity,
%. By definition, then the polynomial mapping $(f_1, f_2)$ is not non-degenerate at infinity, 
which contradicts our assumption.

\subsubsection*{Case 2: $d_1>d_2+\theta$.} For each $j\in J$, we have 
\begin{eqnarray*}
\frac{\partial f_1}{\partial x_j}(\varphi(t))-\lambda(t) \frac{\partial f_2}{\partial x_j}(\varphi(t))
&=& -\lambda^0\frac{\partial f_{2, \Delta_2}}{\partial
x_j}(x^0)t^{d_2+\theta-q_j}  + o(t^{d_2+\theta-q_j}).
%\textrm{ higher order terms in } t. 
\end{eqnarray*}
From $d_2+\theta < d_1\le q_{j_*}$ and (b4), we get $\frac{\partial f_{2, \Delta_2}}{\partial x_j}(x^0) = 0$ for all $j \in J.$ 
On the other hand, a simple calculation shows that
\begin{eqnarray*}
f_2(\varphi(t)) &=& f_{2,\Delta_2}(x^0)t^{d_2} + o(t^{d_2}).
%\textrm{ higher order terms in } t.
\end{eqnarray*}
If $d_2 < 0$, then it follows from (b3) that $f_{2,\Delta_2}(x^0) = 0$
and so $f_2$ is not Khovanskii non-degenerate at infinity, which 
%by definition, the polynomial mapping $(f_1, f_2)$ is not non-degenerate at infinity, which 
contradicts our assumption. If $d_2 =  0,$ we have $f_{2,\Delta_2}(x^0) = r$ and so $r \in K_0(f_{2,\Delta_2}),$ a contradiction. Finally, if $d_2 > 0,$ then $r = 0,$ which contradicts the assumption $|r| > 0.$

\subsubsection*{Case 3: $d_1=d_2+\theta$.} For each $j\in J$, we have 
\begin{eqnarray*}
\frac{\partial f_1}{\partial x_j}(\varphi(t))-\lambda(t)
\frac{\partial f_2}{\partial x_j}(\varphi(t)) &=& \left(\frac{\partial
	f_{1, \Delta_1}}{\partial x_j}(x^0)-\lambda^0\frac{\partial f_{2,
	\Delta_2}}{\partial x_j}(x^0)\right)t^{d_1-q_j} + o(t^{d_1-q_j}). 
\end{eqnarray*}
%where the dots stand for the higher-order terms in $t.$ 
Since $d_1\le \min_{j \in J} q_{j} < 0,$ it follows from (b4) that 
\begin{eqnarray*}
\frac{\partial f_{1, \Delta_1}}{\partial x_j}(x^0)-\lambda^0\frac{\partial f_{2, \Delta_2}}{\partial x_j}(x^0) &=& 0 \quad \textrm{ for all } \quad j \in J.
\end{eqnarray*}
Observe that $\frac{\partial f_{1, \Delta_1}}{\partial x_{j}}(x^0)\ne 0$ for some $j \in J$ since otherwise, we get a contradiction by repeating the arguments in Case~1. Hence the set
\begin{eqnarray*}
X &:=& \{x \in (\mathbb{R}^*)^n \ : \  f_{1,\Delta_1} (x) = 0, \nabla f_{1,\Delta_1} (x)  \ne 0\}
\end{eqnarray*}
is a nonempty semi-algebraic smooth manifold in $\mathbb{R}^n.$ Moreover, $x^0$ is a critical point of $f_{2,\Delta_2}|_X.$ Finally, by a similar argument as Case 2, we can see that either $d_2 < 0$ and $f_{2,\Delta_2}(x^0) = 0$ which contradicts the assumption that the polynomial mapping $(f_1, f_2)$ is Khovanskii non-degenerate at infinity, or $d_2 = 0$ and $f_{2,\Delta_2}(x^0) = r$ which contradicts the assumption  $r\not\in K_0(f_{2,\Delta_2}|_X),$ or $d_2 > 0$ and $r = 0,$ which contradicts the assumption $|r| > 0.$
\end{proof}

The following definition is inspired by \cite[Definition 3.1]{Dinh2012}.
\begin{definition}
Let $g, h \colon \mathbb{R}^n \rightarrow \mathbb{R}$ be polynomial functions. A sequence $\{x^k\}_{k \in \mathbb{N}} \subset \mathbb{R}^n$ with $\|x^k\| \to  +\infty$ is said to be
\begin{enumerate}[{\rm (i)}]
\item a {\em sequence of the first type} if $g(x^k) \to  0$ and $|h(x^k)| \ge \delta$ for some $\delta > 0;$
\item a {\em sequence of the second type} if the sequence $\{g(x^k)\}$ is bounded and $|h(x^k)| \to  +\infty.$
\end{enumerate}
\end{definition}

\begin{lemma}\label{Lemma32}
Let $g, h \colon \mathbb{R}^n \rightarrow \mathbb{R}$ be polynomial functions such that $g^{-1}(0) \subset h^{-1}(0)$ and $\widetilde{K}_\infty(g) = \emptyset.$
Then following two statements hold:
\begin{enumerate}[{\rm (i)}]
\item If $\widetilde{K}_\infty(g|_{\{h = r\}}) = \emptyset$ for all $|r| > 0$ sufficiently small, then there are no sequences of the first type.
\item If $\widetilde{K}_\infty(g|_{\{h = r\}}) = \emptyset$ for all $|r|$ sufficiently large, then there are no sequences of the second type.
\end{enumerate}
\end{lemma}
\begin{proof}
(i) By contradiction, assume that there exist a real number $\delta > 0$ and a sequence $x^k \in \mathbb{R}^n,$ with $\|x^k\| \to +\infty,$ such that
$$g(x^k) \to 0 \quad \textrm{ and } \quad |h(x^k)| \ge \delta.$$
Then $|g(x^k)| > 0$ since $g^{-1}(0) \subset h^{-1}(0).$ By the semi-algebraic Sard Theorem, the set of critical values of $h$ is a finite subset of $\mathbb{R}.$ So we can choose $\delta > 0$ sufficiently small so that each of the level sets $h^{-1}(\pm \delta)$ is either empty or a smooth hypersurface. Furthermore, by assumption, we can suppose that $\widetilde{K}_\infty(g|_{\{h = \pm \delta\}}) = \emptyset.$

Let 
$X := \{x \in \mathbb{R}^n \ : \  |h(x)| \ge \delta\}.$ We have
\begin{eqnarray*}
0 &=& \inf_{x \in X} |g(x)| \ < \ |g(x^k)| \quad \textrm{ for all } \quad k.
\end{eqnarray*}
Applying the Ekeland variational principle (Theorem~\ref{th::ekeland}) to the function 
$$X \rightarrow {\mathbb R}, \quad x \mapsto |g(x)|,$$ 
with data $\epsilon := |g(x^k)| > 0$ and $\lambda := \frac{\|x^k\|}{2} > 0,$ there is a point $y^k$ in $X$ such that the following inequalities hold
\begin{eqnarray*}
&&  |g(y^k)|  \le  |g(x^k)|, \\
&&  \|y^k - x^k\| \le \lambda,\\
&&  |g(y^k)|  \le |g(x)| + \frac{\epsilon}{\lambda} \| x - y^k\| \quad \textrm{ for all } \quad x \in X.
\end{eqnarray*}
We deduce easily that $\lim_{k \to \infty} \|y^k\| = +\infty$ and
$\lim_{k \to \infty} g(y^k) = 0.$ Furthermore, since $g^{-1}(0)
\subset h^{-1}(0)$ and $|h(y^k)| \ge \delta > 0,$ we have $g(y^k) \ne
0$ for all $k.$ Passing to a subsequence and replacing $g$ (resp., $h$)
by $-g$ (resp., $-h$) if necessary, we may assume that to all $k$ the following  conditions hold:
$g(y^k) > 0$ and either $h(y^k) > \delta$ or $h(y^k) = \delta.$ By continuity, $g$ and $h$ are positive in some open neighborhood of $y^k.$ In particular, we have for all $x$ near $y^k,$ 
\begin{eqnarray*}
|g(x)| = g(x) \quad \textrm{ and } \quad |h(x)| = h(x).
\end{eqnarray*}
Hence $y^k$ is a local minimizer of the function
\begin{eqnarray*}
\{x \in \mathbb{R}^n \ : \ h(x) \ge \delta \} \to \mathbb{R}, \quad x \mapsto g(x) + \frac{\epsilon}{\lambda} \| x - y^k\|.
\end{eqnarray*}
Observe that $h^{-1}(\delta)$ is a smooth hypersurface, the function $g$ is smooth and the function $x \mapsto \| x - y^k\|$ is locally Lipschitz. Therefore, by Lagrange's multipliers theorem (see \cite[Theorem~6.1.1]{Clarke1983}), there exists $\mu_k \le 0$ with $\mu_k (h(y^k) - \delta) = 0$ such that
\begin{eqnarray*}
0 & \in & \partial \left (g(\cdot) + \frac{\epsilon}{\lambda}(\|\cdot - y^k\|) \right) (y^k) + \mu_k \nabla h(y^k),
\end{eqnarray*}
where for a locally Lipschitz function $f \colon \mathbb{R}^n \to \mathbb{R},$ the notation $\partial f(x)$ denotes the Clarke derivative of $f$ at $x.$ Using the properties of the Clarke derivative (see \cite[Chapter~2]{Clarke1983}) we derive
\begin{eqnarray*}
\nabla g(y^k) + \mu_k \nabla h(y^k)  & \in & \frac{\epsilon}{\lambda} {\mathbb{B}^n},
\end{eqnarray*}
where $\mathbb{B}^n$ stands for the unit closed ball in $\mathbb{R}^n.$
Consequently, we get
\begin{eqnarray*}
\|\nabla g(y^k) + \mu_k \nabla h(y^k)\| & \le  &  \frac{\epsilon}{\lambda} \ = \ \frac{2|g(x^k)|}{\|x^k\|}.
\end{eqnarray*}
By letting $k$ tend to infinity, we obtain
\begin{equation*}
	\lim_{k \to \infty} \|y^k\| = + \infty, \quad \lim_{k \to \infty} g(y^k) = 0, \quad \textrm{ and } \quad \lim_{k \to \infty} \|\nabla g(y^k) + \mu_k \nabla h(y^k) \| = 0.
\end{equation*}
By passing to a subsequence if necessary, we can assume that either $h(y^k) > \delta$ for all $k$ or $h(y^k) = \delta$ for all $k$. 
%and either
%$\mu_k < \rho<0$
%for all $k$ or $\mu_k = 0$ for all $k$, where $\rho$ is a constant. 
For the former case, $\mu_k = 0$ for all $k$ and hence $0 \in
\widetilde{K}_\infty(g)$; for the latter case, $\|\nabla (g|_{\{h = \delta\}})(y^k)\| \to 0$ and thus $0 \in \widetilde{K}_\infty(g|_{\{h = \delta\}}).$ In both cases we get a contradiction to our assumption.

(ii) Suppose on the contrary that there exists a sequence $x^k \in
\mathbb{R}^n,$ with $\|x^k\| \to +\infty$ such that the sequence
$\{g(x^k)\}$ is bounded and $|h(x^k)| \to +\infty.$ Then $|g(x^k)| >
0$ from our assumption $g^{-1}(0) \subset h^{-1}(0).$ By the semi-algebraic Sard Theorem, the set of critical values of $h$ is a finite subset of
$\mathbb{R}.$ So we can choose $M > 0$ sufficiently large so that each
of the level sets $h^{-1}(\pm M)$ is either empty or a smooth
hypersurface. Furthermore, by assumption, we can suppose that
$\widetilde{K}_\infty(g|_{\{h = \pm M\}}) = \emptyset.$

Let 
$X := \{x \in \mathbb{R}^n \ : \  |h(x)| \ge M \}.$ We have for all $k$ sufficiently large,
\begin{eqnarray*}
&& \inf_{x \in X} |g(x)| \ \le \ |g(x^k)|.
\end{eqnarray*}
By applying the Ekeland variational principle (Theorem~\ref{th::ekeland}) to the function $X \rightarrow \mathbb{R}, x \mapsto |g(x)|,$ with data $\epsilon := |g(x^k)| > 0$ and $\lambda := \frac{\|x^k\|}{2} > 0,$ we get a point $y^k$ in $X$ satisfying the following inequalities 
\begin{eqnarray*}
&&  |g(y^k)|  \le  |g(x^k)|, \\
&&  \|y^k - x^k\| \le \lambda,\\
&&  |g(y^k)|  \le |g(x)| + \frac{\epsilon}{\lambda} \| x - y^k\| \quad \textrm{ for all } \quad x \in X.
\end{eqnarray*}
We deduce easily that 
\begin{eqnarray*}
\frac{\|x^k\|}{2} &\le& \|y^k\|  \ \le \ \frac{3 \|x^k\|}{2},
\end{eqnarray*}
which yields $\lim_{k \to \infty} \|y^k\| = +\infty.$ 

Similarly to (i), for $k$ large enough, we can assume that $h(y^k) \ge
M > 0$ and $g(y^k) > 0$ since $g^{-1}(0) \subset h^{-1}(0)$. Hence, repeating arguments similar to (i), we have
\begin{eqnarray*}
\|\nabla g(y^k) + \mu_k \nabla h(y^k)\| & \le  &  \frac{\epsilon}{\lambda} \ = \ \frac{2|g(x^k)|}{\|x^k\|}
\end{eqnarray*}
for some $\mu_k \le 0$ with $\mu_k(h(y^k) - M) = 0.$ Hence,
\begin{eqnarray*}
\lim_{k \to \infty} \|\nabla g(y^k) + \mu_k \nabla h(y^k)\| &=& 0.
\end{eqnarray*}

On the other hand, since the sequence $\{g(x^k)\}$ is bounded, so is
the sequence $\{g(y^k)\}.$ Hence, by passing to a subsequence if
necessary, we may assume the existence of the limit $t := \lim_{k \to
\infty}  g(y^k).$ Furthermore, we can assume that either $h(y^k) > M$ for all $k$ or $h(y^k) = M$ for all $k.$ 
For the former case, $\mu_k = 0$ for all $k$ and hence $t \in \widetilde{K}_\infty(g)$; for the latter case, $\|\nabla (g|_{\{h = M\}})(y^k)\| \to 0$ and thus $t \in \widetilde{K}_\infty(g|_{\{h = M\}}).$ In both cases we get a contradiction to our assumption.
\end{proof}

\begin{lemma}\label{Lemma33}
Let $g, h \colon \mathbb{R}^n \rightarrow \mathbb{R}$ be polynomial functions such that $g^{-1}(0) \subset h^{-1}(0).$ The following two conditions are equivalent:
\begin{enumerate}[{\rm (i)}]
\item there are no sequences of the first and second types.
\item there exist some constants $c > 0, \alpha > 0,$ and $\beta > 0$ such that
\begin{eqnarray*}
|g(x)|^{\alpha}  + |g(x)|^{\beta}   &\ge& c|h(x)| \quad \textrm{ for all } \quad x \in \mathbb{R}^n.
\end{eqnarray*}
\end{enumerate}
\end{lemma}
\begin{proof}(Cf. \cite[Theorem~3.4]{HaHV2017}.)

(ii) $\Rightarrow$ (i): The implication is straightforward.

(i) $\Rightarrow$ (ii): We assume that $h\not\equiv 0$, otherwise the
implication is trivial.
We only consider the case where $g^{-1}(0) \ne \emptyset;$ the case $g^{-1}(0) = \emptyset$ follows similarly. Then for each $t \ge 0,$ the set $\{x \in \mathbb{R}^n \ : \ |g(x)| = t\}$ is non-empty. This, together with condition~(i), implies that the (semi-algebraic) function $\mu \colon [0, +\infty) \to \mathbb{R}$ given by
\begin{eqnarray*}
\mu(t)  &:=& \sup_{|g(x)| = t} |h(x)|
\end{eqnarray*}
is well-defined. Furthermore, $\mu(0) = 0,$
$\mu(t) > 0$ for all $t > 0$ small enough and $\mu(t) \to +\infty$ as $t \to +\infty.$ By  Lemma~\ref{GrowthDichotomyLemma}, we
can write 
\begin{eqnarray*}
\mu(t)  &=& a t^{\alpha} + o(t^\alpha)  \quad \textrm{ as } \quad t \to 0^+,\\
\mu(t)  &=& b t^{\beta} + o(t^\beta)  \quad \textrm{ as } \quad t \to +\infty
\end{eqnarray*}
for some constants $a \ne 0, b \ne 0, \alpha \ge 0$ and $\beta > 0.$ 
Therefore, we can find constants $c_1 > 0, c_2 > 0, \delta > 0$ and $r > 0$ with $\delta \ll 1 \ll r$ such that the following inequalities hold
\begin{eqnarray*}
|g(x)|^{\alpha} &\ge& c_1|h(x)| \quad \textrm{ for } \quad 0 < |g(x) | \le \delta, \\
|g(x)|^{\beta} &\ge& c_2|h(x)| \quad \textrm{ for } \quad |g(x) | \ge r.
\end{eqnarray*}
By assumption, we may assume that $\alpha > 0$ so that the first inequality holds for $|g(x)| \le \delta.$ Furthermore, we also may assume $\alpha \le 1 \le \beta$ because $\delta$ is sufficiently small and $r$ is sufficiently large.

On the other hand, it follows easily from condition~(i) that there exists a constant $M > 0$ such that for all $x \in \mathbb{R}^n$ with $\delta \le |g(x)| \le R$ we have $|h(x)| \le M$ and hence 
\begin{eqnarray*}
|g(x)|^{\alpha} + |g(x)|^{\beta} 
& \ge & \delta^\alpha + \delta^\beta \ = \ \frac{\delta^\alpha + \delta^\beta}{M}M 
\ \ge \ \frac{\delta^\alpha + \delta^\beta}{M}\, |h(x)|.
\end{eqnarray*}
Letting $c := \min\{c_1, c_2, \frac{\delta^\alpha + \delta^\beta}{M}\},$ we get the desired conclusion.
\end{proof}

We now are in position to finish the proof of Theorem~\ref{Theorem11}.

\begin{proof}[Proof of Theorem~\ref{Theorem11}]
This is a direct consequence of  Lemmas~\ref{Lemma31-1}, \ref{Lemma31-2}, \ref{Lemma32}, and \ref{Lemma33}.
\end{proof}

The following corollary is inspired by Theorems~1.1 and 1.3 in \cite{Fernando2014-2}.
\begin{corollary}
Under the assumptions of Theorem~\ref{Theorem11}, there exist a positive integer $N$ and a continuous semi-algebraic function $f \colon \mathbb{R}^n \to \mathbb{R}$ such that $h^N = g f.$
\end{corollary}

\begin{proof}
Clearly, if $\inf_{x \in \mathbb{R}^n} |g(x)| > 0$ then the integer $N := 1$ and the function $f := \frac{h}{g}$ have the desired property. So assume that $\inf_{x \in \mathbb{R}^n} |g(x)| = 0.$  By observing the proof of Lemma~\ref{Lemma33}, we can find positive constants $\alpha$ and $\delta$ with $\alpha \le 1$ such that
\begin{eqnarray*}
|g(x)|^{\alpha} &\ge& c|h(x)| \quad \textrm{ for } \quad |g(x)| \le \delta.
\end{eqnarray*}

Let $\ell := \left[\frac{1}{\alpha}\right] + 1 > \frac{1}{\alpha} \ge
1.$ The following function $f_0 \colon \mathbb{R}^n \to \mathbb{R}$ defined by
\begin{eqnarray*}
f_0(x) &:=&
\begin{cases}
\frac{h^{2\ell}(x)}{g^2(x)} & \textrm{ if } g(x) \ne 0,\\
0  & \textrm{ otherwise}
\end{cases}
\end{eqnarray*}
is continuous and semi-algebraic. Since $f_0g^2 = h^{2 \ell},$ we deduce that the integer $N := 2\ell$ and the function $f := f_0 g$ have the desired property.
\end{proof}

\section{Genericity of non-degenerate at infinity polynomial mappings}\label{SectionGenericity}

In this section we show the genericity of the condition of
non-degeneracy at infinity for real polynomial mappings (Theorem \ref{Theorem12}); actually, we will prove a stronger result (see Theorem~\ref{OpenDenseTheorem} below). Note that the genericity of the condition of non-degeneracy for complex polynomial mappings has been given in \cite{Kouchnirenko1976} (the case $p = 1$) and in \cite[Theorem (Resolution of Singularities)]{Khovanskii1978} and \cite[Corollary~3.2.1]{Oka1997} (the case $p \ge 1$).

For simplicity, we introduce some notation here for this section. Let
$\Gamma := (\Gamma_1, \ldots, \Gamma_p)$ with $1 \le p \le n$ and each $\Gamma_i$ being a Newton polyhedron in $\mathbb R_+^n.$
%Recall the notation $\Delta(q,\Gamma_i)$ introduced in \eqref{eq::delta} 
Let
\begin{eqnarray*}
\mathcal{F} &:=& \{\Delta := (\Delta_1,\ldots,\Delta_p) \ : \ \exists q \in \mathbb{R}^n \textrm{ s.t. } \Delta_i = \Delta(q,\Gamma_i) \textrm{ for all } i\}. 
\end{eqnarray*}
%(Recall that $\Delta(q,\Gamma_i) := \textrm{argmin}_{\kappa\in\Gamma_i} \langle q, \kappa\rangle.$) 
Clearly, $\mathcal{F}$ is a finite set as the number of faces of a polyhedron is finite.

For each $i=1,\ldots,p$, let
$\mathcal{Z}_i:=\Gamma_i\cap\mathbb{Z}^n$ and $m_i :=
\#\mathcal{Z}_i$-the number of points in the set $\mathcal{Z}_i.$ 
For each polynomial mapping from $\mathbb R^n$ to $\mathbb R^p$ such that the Newton polyhedra of its components are given by $\Gamma$, we can index the coefficients of the $i\text{th}$ component over the integer points of $\Gamma_i.$
So let $c_i = (c_{i, \kappa})_{\kappa\in\mathcal{Z}_i}$,
$$f_i (x, c_i) := \sum_{\kappa\in\mathcal{Z}_i}c_{i, \kappa} x^\kappa\in {\mathbb R}[x]\ \text{ and }\ F(x,c):=(f_1(x,c_1),\dots,f_p(x,c_p)).$$
Denote by $\mathbb{R}^{\Pi}$ the product space ${\mathbb{R}}^{m_1} \times \cdots \times {\mathbb{R}}^{m_p}$.
For any $c := (c_1, \ldots, c_p)  \in \mathbb{R}^{\Pi}$, set 
\begin{eqnarray*}
\mathcal{G}(c):=\left(\Gamma(f_1 (x, c_1)), \ldots, \Gamma(f_p (x, c_p))\right).
\end{eqnarray*}
For each nonempty set $I := \{i_1, \ldots, i_s\}\subset\{1, \ldots, p\}$ and $\Delta := (\Delta_1,\ldots,\Delta_p)\in\mathcal{F},$ we let
\begin{eqnarray*}
F_{I, \Delta}(x, c) &:=& (f_{i_1, \Delta_{i_1}}(x, c_{i_1}), \ldots, f_{i_s, \Delta_{i_s}}(x, c_{i_s})), \\
x D F_{I, \Delta}(x, c) &:=& \left (x_j \frac{\partial f_{i,
\Delta_i}}{\partial x_j}(x, c_i) \right)_{i \in I,\ j=1,\ldots,n},\\
\mathcal{V}(I, \Delta, c) &:=& \left\{x\in(\mathbb{R}^*)^n : F_{I, \Delta}(x,
c) = 0\right\},\\
\mathcal{V}_{\text{\upshape reg}}(I, \Delta, c) &:=& \left\{x\in \mathcal V(I,
	\Delta, c) : \textrm{rank}(xDF_{I, \Delta}(x,c)) = \#I \right\},\\
	\mathcal {D}_{I}(\Delta) &:=& \left\{c\in\mathbb{R}^{\Pi}\ :\  \mathcal{G}(c)=\Gamma,\	
\mathcal{V}(I, \Delta, c)=\mathcal{V}_{\text{\upshape reg}}(I, \Delta,
c) \right\}.
\end{eqnarray*}
%and
%\[
%\mathcal {D}_{I}(\Delta) := \left\{c\in\mathbb{R}^{\Pi}\ :\  \mathcal{G}(c)=\Gamma,\	
%\mathcal{V}(I, \Delta, c)=\mathcal{V}_{\text{\upshape reg}}(I, \Delta, c)
%\right\}.
%\]

With the notation above, Theorem \ref{Theorem12} is a direct consequence of the following.
\begin{theorem} \label{OpenDenseTheorem}
The set $\cap_{I, \Delta} \mathcal{D}_I(\Delta)$ is an open dense
semi-algebraic set in $\mathbb{R}^{\Pi},$
where the intersection is taken over all nonempty sets $I \subset \{1, \ldots, p\}$ and all $\Delta \in \mathcal{F}.$
\end{theorem}

\begin{proof}
Observe that the number of subsets of $\{1, \ldots, p\}$ is finite, $\mathcal{F}$ is a finite set, and a finite intersection of open dense semi-algebraic sets is open dense semi-algebraic. Now the desired conclusion follows
immediately from Propositions~\ref{Proposition41}~and~\ref{Proposition42} below.
\end{proof}

\begin{proposition} \label{Proposition41}
For each nonempty set $I \subset \{1, \ldots, p\},$ the set $\cap_{\Delta \in \mathcal{F}} \mathcal{D}_I(\Delta)$ is open and semi-algebraic.
\end{proposition}

\begin{proof} (Cf. \cite[Appendix]{Oka1979}; see also \cite[Proposition~3.1]{Chen2012}.)
Let $I$ be a nonempty subset of $\{1, \ldots, p\}.$ By renumbering, we may assume that $I = \{1, \ldots, s\}$ for some $s \le p.$ By definition, for any $(\Delta_1,\ldots,\Delta_p) \in\mathcal{F}$ we have
\begin{eqnarray*}
\mathcal {D}_{I}(\Delta_1,\ldots,\Delta_p) &=& \mathcal {D}_{I}(\Delta_1, \ldots, \Delta_s) \times X
\end{eqnarray*}
where
$$X := \{(c_{s + 1}, \ldots, c_p) \ : \ \Gamma(f_i(x, c_i)) =
\Gamma_i \textrm{ for } i = s + 1, \ldots, p\}.$$
Observe that $X$ is
an open dense semi-algebraic subset of
$\mathbb{R}^{m_{s+1}}\times\cdots\times\mathbb{R}^{m_p}$ and that $X$
does not depend on the polyhedra $\Gamma_i$ for $i = 1, \ldots, s.$ Hence,
it suffices to show that
$$\cap_{(\Delta_1, \ldots, \Delta_s)} \mathcal{D}_I(\Delta_1, \ldots, \Delta_s)$$
is an open semi-algebraic subset of ${\mathbb{R}}^{m_1} \times \cdots \times
{\mathbb{R}}^{m_s}.$ In other words, we can assume that $s = p,$ i.e.,
$I=\{1, \ldots, p\}.$

Consider the projection
$$\pi \colon {\mathbb R}^n  \times \mathbb{R}^{\Pi}  \rightarrow
{\mathbb{R}}^{\Pi}, \quad (x, c) \mapsto c,$$
and the union $V^* := \cup_{\Delta \in \mathcal{F}}V(\Delta)$ where
$$V(\Delta) := \left\{
(x, c)\in {\mathbb R}^n  \times \mathbb{R}^{\Pi}
\ : \ \mathcal{G}(c)=\Gamma,\ x\in\mathcal{V}(I, \Delta, c)\setminus\mathcal{V}_{\text{\upshape reg}}(I,
\Delta, c)\right\}.$$
By definition, $W := \pi(V^*)$ is the complement of $\cap_{\Delta \in
\mathcal{F}} \mathcal{D}_I(\Delta)$ in the set
$\{c\in\mathbb{R}^{\Pi}\ :\  \mathcal{G}(c)=\Gamma\}.$
Observe that the latter set is an open dense semi-algebraic subset of
${\mathbb{R}}^{\Pi}.$
In light of Theorem~\ref{TarskiSeidenbergTheorem1}, $W$ is a
semi-algebraic set, and so is $\cap_{\Delta \in \mathcal{F}}
\mathcal{D}_I(\Delta).$

Next we show that $W$ is closed, or equivalently, $W = \overline{W}.$ To see this, take a point  $c^0:=(c_1^0, \ldots, c_p^0) \in \overline{W}.$ 
We show $c \in W.$ Indeed, if $c^0$ is an isolated point of $\overline{W},$ then
$c^0 \in W$ and we are done. So assume that $c^0$ is not isolated in $\overline{W}.$
By definition, $c^0\in\overline{\pi(V(\Delta))}$ for some $\Delta :=
(\Delta_1,\ldots,\Delta_p) \in \mathcal{F}.$ In view of
Lemma~\ref{CurveSelectionLemma1},
there exists a non-constant real analytic mapping
$t \mapsto (\varphi(t), c(t)) \in V(\Delta)$
defined on a small enough interval $(0, \epsilon)$ such that $\lim_{t
\to 0} c(t) = c^0.$ Let us expand $\varphi_j(t), j = 1, \ldots, n,$
and $c_i(t)$, $i=1,\ldots,p$, in terms of the parameter, say
\begin{eqnarray*}
\varphi_j(t) =  x_j^0 t^{q_j} + o(t^{q_j})\quad\text{and}\quad
c_i(t) = c_i^0  + o(1),
\end{eqnarray*}
where $x_j^0 \ne 0$ and $q_j \in {\mathbb Q}$ (note that for all $t$ we have $\varphi(t) \in (\mathbb{R}^*)^n$ and so $\varphi_j(t) \ne 0).$  
Let $q := (q_1,\ldots,q_n)$ and
\begin{eqnarray*}
\widetilde{\Delta}_i & := & \Delta(q, \Delta_i) \quad \textrm{ for all } \quad i = 1, \ldots, p.
\end{eqnarray*}
We prove that $\widetilde{\Delta} := (\widetilde{\Delta}_1,\ldots,\widetilde{\Delta}_p)$ belongs to $\mathcal{F}.$
In fact, if $q = 0,$ then $\widetilde{\Delta} = {\Delta} \in \mathcal{F}$ and there is nothing to prove.  So, assume that $q \ne 0.$ By
definition, we can find a vector $q^0 \in \mathbb{R}^n$ such that
$\Delta_i = \Delta(q^0, \Gamma_i)$ for all $i$. If $\Delta_i=\Gamma_i$
for all $i$, then it is clear that $\widetilde{\Delta} \in
\mathcal{F}.$ Otherwise, we have $q^0 \ne 0$ and $I' :=
\{i\in\{1,\ldots,p\} \ : \  \Delta_i \neq \Gamma_i\} \ne \emptyset.$
Then for each $i\in I'$, there exists $\epsilon_i>0$ such that for
any $\widetilde q$ with $\|\widetilde q - q^0\| \le \epsilon_i$, we
get that
\begin{eqnarray*}
\Delta(\widetilde q,\Gamma_i) &\subset& \Delta(q^0,\Gamma_i) \ = \ \Delta_i.
\end{eqnarray*}
Set $\epsilon := \min_{i\in I'}\epsilon_i>0$ and $\widetilde q := q^0+\epsilon \frac{q}{\|q\|}.$ Clearly $\Delta(\widetilde q,\Gamma_i)\subset\Delta_i$. Hence $\Delta(\widetilde q,\Gamma_i)=\Delta(\widetilde q,\Delta_i)$. Moreover, for any $\kappa\in\Delta_i,$ we have
\begin{eqnarray*}
\langle \widetilde q,\kappa\rangle &=& \langle q^0,\kappa\rangle+\frac{\epsilon}{\|q\|}\langle q,\kappa\rangle \ \ge \ d_i+\frac{\epsilon}{\|q\|} \widetilde{d}_i,
\end{eqnarray*}
where $d_i := \min_{\kappa' \in \Gamma_i}\langle q^0, \kappa'\rangle$ and  $\widetilde{d}_i := \min_{\kappa' \in \Delta_i}\langle q, \kappa'\rangle.$
Observe that the equality happens if and only if $\kappa\in\widetilde\Delta_i$, which yields $\widetilde\Delta_i = \Delta(\widetilde q,\Delta_i).$ Therefore $\widetilde\Delta_i=\Delta(\widetilde q,\Gamma_i)$ and so $\widetilde{\Delta} \in \mathcal{F}.$

On the other hand, a simple calculation shows that
\begin{eqnarray*}
f_{i, \Delta_i} (\varphi(t),c_i(t)) &=& f_{i,
\widetilde{\Delta}_i}(x^0, c^0_i)t^{d_i} + o(t^{d_i}),\\
\varphi_j(t) \frac{\partial f_{i, \Delta_i}}{\partial x_j}(\varphi(t),c_i(t))
&=& x_j^0 \frac{\partial f_{i, \widetilde{\Delta}_i}}{\partial
x_j}(x^0, c^0_i)t^{d_i} + o(t^{d_i}),
\end{eqnarray*}
where $x^0 := (x_1^0, \ldots, x_n^0) \in ({\mathbb R}^*)^{n}.$  As  $(\varphi(t), c(t)) \in V(\Delta)$ for all
$t\in(0, \epsilon),$ we get easily that $(x^0, c^0) \in V(\widetilde{\Delta}) \subset V^*.$
Thus $c^0 \in W,$ as required.
\end{proof}

\begin{proposition} \label{Proposition42}
For each nonempty set $I \subset \{1, \ldots, p\}$ and each $\Delta \in \mathcal{F},$ $\mathcal{D}_I(\Delta)$ contains an open dense semi-algebraic subset of $\mathbb{R}^{\Pi}.$
\end{proposition}

Before proving Proposition~\ref{Proposition42}, we need the following auxiliary lemma.
\begin{lemma}\label{Integer} 
Let $d$ and $n$ be integers such that $0\leqslant d\leqslant n.$ 
Let $q^1,\dots,q^{n-d}$ be linearly independent vectors in $\mathbb Z^n$ and $S\subset\mathbb Z_+^n.$
Assume that $\langle q^j ,\kappa \rangle\geqslant 0$ for all $\kappa\in S$ and $j=1,\dots,n-d$.
Then the following statements hold:
\begin{enumerate}[{\rm (i)}]
\item There are $q^{n-d+1},\dots,q^n\subset\{e^1,\dots,e^n\}$ such that the system $\{q^1,\dots,q^{n}\}$ is linearly independent.
\item There exist vectors $\widetilde q^1,\dots,\widetilde q^{n}\in \mathbb Z^n$ such that
\begin{enumerate}%[{\rm (1)}]
\item [\rm (ii1)] $\span\{\widetilde q^1,\dots,\widetilde q^{k}\}=\span\{q^1,\dots,q^{k}\}$ for $k=1,\dots,n$;
\item [\rm (ii2)] $\langle \widetilde q^j ,\kappa \rangle\geqslant 0$ for all $\kappa\in S$ and $j=1,\dots,n$;
\item [\rm (ii3)] $\co\{0,\widetilde q^1,\dots,\widetilde q^{n}\} \cap\mathbb Z^n = \{0,\widetilde q^1,\dots,\widetilde q^{n}\}$; and
\item [\rm (ii4)] $|\det(\widetilde q^1,\dots,\widetilde q^{n})|=1.$ 
\end{enumerate}
Here for a set $X \subset \mathbb{R}^n,$ the notation $\span X$ denotes the smallest linear subspace containing $X$ and the notation 
$\co(X)$ denotes the convex hull of $X.$

\end{enumerate}
\end{lemma}
\begin{proof} 
As (i) is clear, it remains to prove (ii).
Clearly $\langle q^j ,\kappa \rangle\geqslant 0$ for all $\kappa\in S$ and $j=n-d+1,\dots,n$. 
The proof is done by induction as follows. 

Set $\displaystyle\widetilde q^1:=\frac{q^1}{N}$ where $N\geqslant 1$ is the integer such that $\displaystyle\frac{q^1}{N}\in\mathbb Z^n$ and $[0,\frac{q^1}{N}]\cap\mathbb Z^n=\{0,\frac{q^1}{N}\}.$ For $1\leqslant k < n,$ assume that we have constructed a linearly independent system $\{\widetilde q^1,\dots,\widetilde q^{k}\}$ such that
\begin{itemize}
\item $\span\{\widetilde q^1,\dots,\widetilde q^{k}\}=\span\{q^1,\dots,q^{k}\}$; 
\item $\langle \widetilde q^j ,\kappa \rangle\geqslant 0$ for all $\kappa\in S$ and $j=1,\dots,k$; and
\item $\co\{0,\widetilde q^1,\dots,\widetilde q^{k}\}\cap\mathbb Z^n=\{0,\widetilde q^1,\dots,\widetilde q^{k}\}.$
\end{itemize}
If $\co\{0,\widetilde q^1,\dots,\widetilde q^k,q^{k+1}\}\cap\mathbb Z^n=\{0,\widetilde q^1,\dots,\widetilde q^k, q^{k+1}\},$ then we are done by setting $\widetilde q^{k+1}= q^{k+1}.$
Otherwise, set 
$$N_1:=\#((\co\{0,\widetilde q^1,\dots,\widetilde q^k,q^{k+1}\}\cap\mathbb Z^n)\setminus\{0,\widetilde q^1,\dots,\widetilde q^k, q^{k+1}\})>0.$$
Let $a^1\in (\co\{0,\widetilde q^1,\dots,\widetilde q^k,q^{k+1}\}\cap\mathbb Z^n)\setminus\{0,\widetilde q^1,\dots,\widetilde q^k, q^{k+1}\}.$
Then $$a^1=\sum_{l=1}^{k} t_l \widetilde q^l+t_{k+1} q^{k+1}\ \text{ with }\ 0\leqslant t_1,\dots,t_{k+1} <1.$$
Observe that $t_{k+1}>0$ since otherwise 
$a^1\in (\co\{0,\widetilde q^1,\dots,\widetilde q^k\}\cap\mathbb Z^n)\setminus\{0,\widetilde q^1,\dots,\widetilde q^k\},$ which is a contradiction.
Therefore the system $\{\widetilde q^1,\dots,\widetilde q^{k},a^1\}$ is linearly independent and 
$$\span\{\widetilde q^1,\dots,\widetilde q^{k},a^1\}=\span\{q^1,\dots,q^{k},q^{k+1}\}.$$ 
In addition, for all $\kappa\in S$, we have
$$\langle a^1 ,\kappa \rangle=\left\langle \sum_{l=1}^{k} t_l \widetilde q^l+t_{k+1} q^{k+1} ,\kappa \right\rangle=\sum_{l=1}^{k} t_l \langle\widetilde q^l,\kappa\rangle+t_{k+1}\langle q^{k+1},\kappa\rangle\geqslant 0.$$
Let 
$$N_2:=\#((\co\{0,\widetilde q^1,\dots,\widetilde q^k,a^{1}\}\cap\mathbb Z^n)\setminus\{0,\widetilde q^1,\dots,\widetilde q^k, a^{1}\}).$$
Clearly $N_2<N_1.$ 
If $N_2=0$, set $\widetilde q^{k+1}= a^{1}$ and we are done.
Otherwise, since $N_2$ is finite, by repeating the procedure finitely many times, we must get the vector $\widetilde q^{k+1}$ such that 
\begin{itemize}
\item $\span\{\widetilde q^1,\dots,\widetilde q^{k+1}\}=\span\{q^1,\dots,q^{k+1}\}$; 
\item $\langle \widetilde q^j ,\kappa \rangle\geqslant 0$ for all $\kappa\in S$ and $j=1,\dots,k+1$; and
\item $\co\{0,\widetilde q^1,\dots,\widetilde q^{k+1}\}\cap\mathbb Z^n=\{0,\widetilde q^1,\dots,\widetilde q^{k+1}\}.$
\end{itemize}
By induction, (ii1),~(ii2) and~(ii3) follow. By applying~\cite[Exercises of page 48]{Fulton1993}, (ii4) follows from~(ii3). The lemma is proved.
\end{proof}

\begin{proof} [Proof of Proposition~\ref{Proposition42}]
As in the proof of Proposition~\ref{Proposition41}, we may assume that
$I =\{1, \ldots, p\}.$ Furthermore, observe that for any
$(c_1,\ldots,c_p) \in \mathcal{D}_I(\Delta),$ all coefficients
$c_{i,\kappa},$ with $\kappa \in \mathcal{Z}_i 
\setminus \Delta_i$ and $i = 1,\ldots, p,$ can be replaced by any
nonzero real numbers and the resulting $(c_1,\ldots,c_p)$ still belong to $\mathcal{D}_I(\Delta).$ Consequently, without loss of generality, we may assume that $\Delta_i = \Gamma_i$ for all $i.$ In other words, we need to show that the set
\begin{equation}\label{DIG}
	\mathcal {D}_{I}(\Gamma) = \left\{c\in\mathbb{R}^{\Pi}\ :\  \mathcal{G}(c)=\Gamma,\	
\mathcal{V}(I, \Gamma, c)=\mathcal{V}_{\text{\upshape reg}}(I, \Gamma,
c) \right\},
\end{equation}
%\[
%\mathcal{D}_I(\Gamma) = \left\{\begin{array}{lll}
%(c_1, \ldots, c_p)
%\ : \ \left\{ 
%	\begin{aligned}
%		&
%		c_i=(c_{i,\kappa})_{\kappa\in\mathcal{Z}_i}\in\mathbb{R}^{m_i},
%		\ \Gamma(f_i(x,c_i))=\Gamma_i \textrm{ for } i = 1, \ldots, p,\\
%		&\textrm{if } x\in(\mathbb{R}^*)^n \textrm{ and if } F(x, c) = 0, \textrm{ then }\\
%		&\textrm{rank}(x DF(x, c)) = p 
%	\end{aligned}\right.
%\end{array}\right\},
%\]
contains an open dense semi-algebraic subset of $\mathbb{R}^{\Pi}.$ 
%where we put
%\begin{eqnarray*}
%F(x, c) &:=& (f_1(x,c_1), \ldots, f_p(x,c_p)), \\
%x D F(x,c) &:=&  \left (x_j \frac{\partial f_i}{\partial x_j}(x,c_i) \right)_{i=1,\ldots,p,\ j=1,\ldots,n}.
%\end{eqnarray*}

It is clear that if there exists an index $i_0$ such that $m_{i_0} = 1$ (i.e., $f_{i_0}(x, c_{i_0})$ is a monomial), then $\{f_{i_0}(x, c_{i_0}) = 0\} \subset \{x_1 \cdots x_n=0\}$ for $c_{i_0}\ne 0;$ hence $(\mathbb{R}^*)^{m_1}\times \cdots \times (\mathbb{R}^*)^{m_p} \subset \mathcal{D}_I(\Gamma)$ and the problem is trivial.  So in what follows we will assume that $m_{i} > 1$ for every $i = 1, \ldots, p.$
Let $\Gamma_1+\dots+\Gamma_p$ be the Minkowski sum and set $ d := \dim(\Gamma_1+\dots+\Gamma_p).$
There are two cases to consider:

\subsubsection*{Case 1: $d = n$} 
For each $i = 1, \ldots, p,$ let $v^{i1}, \ldots, v^{ir_i}$ be the vertices of $\Gamma_i$. Note that $r_i>1$ for every $i$ by the assumption $m_i > 1.$ % Without loss of generality, assume that $r_i=m_i.$ 
Let  
\begin{eqnarray*}
w^{ij} &:=& v^{ij} - v^{ir_i} \quad \textrm{ for } \quad j = 1,\ldots, r_i - 1.
\end{eqnarray*}
%The following lemma follows easily.
%\begin{lemma}\label{Lemma42} 
It is not hard to see that
\begin{eqnarray}\label{Lemma42}
\mathrm{rank} \{w^{11}, \ldots, w^{1(r_1 - 1)}, \ldots, w^{p1}, \ldots, w^{p(r_p-1)}\} = \dim(\Gamma_1+\cdots+\Gamma_p)=n.
\end{eqnarray}
%\end{lemma}
%
%\begin{proof} 
%Let $a,b$ be two arbitrary points of $\Gamma_1+\cdots+\Gamma_p$.  There exist $a^i, b^i \in \Gamma_i, i = 1, \ldots, p,$ such that
%\begin{eqnarray*}
%a &=& a^1+ \cdots + a^p, \\
%b &=& b^1+ \cdots + b^p.
%\end{eqnarray*}
%Observe that 
%\begin{eqnarray*}
%a^i &=& \sum_{j=1}^{r_i} \lambda_{ij} v^{ij}, \quad \text{ with } \quad \lambda_{ij}\ge 0 \quad \textrm{and } \quad \sum_{j=1}^{r_i}\lambda_{ij}=1, \\
%b^i &=& \sum_{j=1}^{r_i} \mu_{ij} v^{ij}, \quad \text{ with } \quad \mu_{ij}\ge 0 \quad \textrm{and } \quad \sum_{j=1}^{r_i}  \mu_{ij}=1.
%\end{eqnarray*}
%Hence
%\begin{eqnarray*}
%b^i - a^i 
%\ = \  \sum_{j=1}^{r_i} (\mu_{ij}  - \lambda_{ij}) v^{ij} 
%&=& \sum_{j=1}^{r_i - 1} (\mu_{ij}  - \lambda_{ij}) v^{ij} + (\mu_{ir_i}  - \lambda_{ir_i}) v^{ir_i} \\
%&=& \sum_{j=1}^{r_i - 1} (\mu_{ij}  - \lambda_{ij}) v^{ij} - \sum_{j=1}^{r_i - 1} (\mu_{ij}  - \lambda_{ij}) v^{ir_i} \\
%&=& \sum_{j=1}^{r_i - 1} (\mu_{ij}  - \lambda_{ij}) w^{ij}.
%\end{eqnarray*}
%Consequently, $b - a = \sum_{i = 1}^p (b^i - a^i)$ is a linear combination of vectors $w^{ij}.$ Since $a,b$ can be chosen arbitrarily, the lemma follows.
%\end{proof}

Now we proceed by induction on $p$, the number of polynomials. 
In what follows, $c_{i, j}$ stands for the coefficient of the monomial $x^{v^{ij}}$ in $f_i(x, c_i).$ 

Firstly, let $p = 1$ and consider the semi-algebraic mapping
$$\Phi \colon (\mathbb{R}^*)^n   \times \mathbb{R}^{m_1} \to \mathbb{R}^{n+1}, \quad  (x, c_1)	   \mapsto \left(x_1\frac{\partial f_1}{\partial x_1}(x, c_1), \ldots, x_n \frac{\partial f_1}{\partial x_n}(x, c_1), f_1(x,c_1)\right).$$
The Jacobian matrix $D\Phi$ of $\Phi$ contains the following matrix
\begin{equation*}
\frac{\partial \Phi}{\partial (c_{1, 1}, \ldots, c_{1, r_1})} = 
\begin{pmatrix}
x^{v^{11}}v^{11} & \cdots & x^{v^{1 (r_1-1)}}v^{1 (r_1-1)} & x^{v^{1r_1}}v^{1r_1}\\
x^{v^{11}}       & \cdots & x^{v^{1 (r_1-1)}}			& x^{v^{1r_1}}
\end{pmatrix},
\end{equation*}
where $v^{1j}\ (j=1,\dots,r_1)$ are written as column vectors. The rank of $\dis\frac{\partial \Phi}{\partial (c_{1, 1}, \ldots, c_{1, r_1})}$ is equal to the rank of the following matrix
$$M_1 : = 
\begin{pmatrix}
v^{11} & \cdots & v^{1 (r_1-1)} & v^{1r_1} \\
1      & \cdots & 1 		    & 1
\end{pmatrix}.$$
%By some linear operations on the columns of $M_1$, we obtain 
The following matrix has the same rank as $M_1$
$$M_2: = 
\begin{pmatrix} 
v^{11}-v^{1r_1} & \cdots & v^{1 (r_1-1)} - v^{1r_1} & v^{1r_1} \\
0               & \cdots & 0                    & 1
\end{pmatrix} =
\begin{pmatrix} 
w^{11} & \cdots & w^{1 (r_1-1)} & v^{1r_1} \\
0      & \cdots & 0           & 1
\end{pmatrix}.$$
In light of~\eqref{Lemma42}, 
%we know that
%$$\mathrm{rank}\{w^{11},\ldots,w^{1 (r_1-1)}\}=\dim\Gamma_1 = d = n.$$
%Hence 
$\mathrm{rank} M_2 = n + 1,$ and so $\mathrm{rank}(D\Phi) = n + 1.$ Consequently  $\Phi\pitchfork\{0\}$ in $\mathbb{R}^{n+1}.$ By Theorem~\ref{SardTheorem}, the set 
$$P_1:=\{c_1\in \mathbb{R}^{m_1} \ : \  \Phi(\cdot, c_1)\pitchfork\{0\}\}$$
contains an open dense semi-algebraic subset of $\mathbb{R}^{m_1}.$ Observe that the mapping $\Phi(\cdot, c_1) \colon (\mathbb{R}^*)^n\to\mathbb{R}^{n+1}$ is transversal to $\{0\}$ if and only if $\mathrm{Im} \Phi(\cdot, c_1)\cap\{0\} = \emptyset.$ Hence, we have $\{\Phi(\cdot, c_1)=0\}=\emptyset$ for $c_1\in P_1.$ Consequently, $P_1 \subset \mathcal{D}_I(\Gamma),$ which completes the proof for the case $p = 1.$

Now assume that $p > 1.$ Recall that $I=\{1,\dots,p\}.$ By induction, for each $l = 1,\ldots, p,$ the set $\mathcal{D}_{I \setminus \{l\}}(\Gamma)$ contains an open dense semi-algebraic set $\widetilde U_l$ in 
$$\widehat{\mathbb{R}}^{\Pi}_l:=\mathbb{R}^{m_1}\times\cdots\times\mathbb{R}^{m_{l-1}}\times\mathbb{R}^{m_{l+1}}\times\cdots\times\mathbb{R}^{m_{p}}.$$ 
Let $U_l\subset\mathbb{R}^{\Pi}$ be the set obtained from $\widetilde U_l\times\mathbb R^{m_l}$ by the following permutation of coordinates
$$\widehat{\mathbb{R}}^{\Pi}_l\times \mathbb{R}^{m_{l}}\to \mathbb{R}^{\Pi}, \ \ (c_1,\dots,c_{l-1},c_{l+1},\dots,c_p,c_l)\mapsto(c_1,\dots,c_p). $$
Consider the semi-algebraic mapping
\begin{eqnarray*}
\Psi \colon (\mathbb{R}^*)^n  \times   ({U}_1 \cap  \cdots \cap  {U}_p)  \times (\mathbb{R}^p-\{0\}) & \to  & \mathbb{R}^{n}  \times  \mathbb{R}^p, \\
(x, c_1, \ldots, c_p, \lambda) & \mapsto & \left(\dis\sum_{i=1}^p \lambda_i x\nabla f_i(x, c_i)  ,  f_1(x, c_1),\ldots,f_p(x, c_p)\right),
\end{eqnarray*}
where, for simplicity of notation, we let 
\begin{eqnarray*}
x\nabla f_i(x,c_i) &:=& \left(x_1\frac{\partial f_i}{\partial x_1}(x,c_i), \ldots,x_n\frac{\partial f_i}{\partial x_n}(x,c_i)\right).
\end{eqnarray*}
Note that if $(x, c, \lambda) \in \Psi^{-1}(0)$, then $\lambda_1
\cdots \lambda_p\not=0$, since if $\lambda_{l}=0 $ for some $l,$ then 
$$\sum_{i \ne l} \lambda_i x\nabla f_i(x, c_i) = 0,$$ 
which implies that $(c_1,\dots,c_{l-1},c_{l+1},\dots,c_p)\notin\widetilde U_l.$
Hence $(c_1,\dots,c_{p})\notin U_l$ which is a contradiction.
The Jacobian matrix $D\Psi$ of $\Psi$ contains the matrix 
\begin{eqnarray*}
M_3 &:=& \frac{\partial \Psi}{\partial [(c_{1, 1}, \ldots, c_{1, r_1}), \ldots, (c_{p, 1}, \ldots, c_{p, r_p})]} 
\ = \ \left(\begin{array}{c|c|c|c} B_1 & B_2 & \cdots & B_p\end{array}\right),
\end{eqnarray*}
where
%\begin{eqnarray*}
%B_1 &:=& \left(
%\begin{array}{cccc}
%\lambda_1 x^{v^{11}}v^{11} & \cdots & \lambda_1 x^{v^{1 (r_1-1)}}v^{1 (r_1-1)} & \lambda_1 x^{v^{1r_1}}v^{1r_1}\\
%x^{v^{11}}       & \cdots  & x^{v^{1(r_1-1)}}	 & x^{v^{1r_1}}\\
%0                & \cdots  & 0                    & 0                      \\
%\vdots           & \cdots  & \vdots               & \vdots                \\
%0                & \cdots  & 0                    & 0                     \\
%\end{array}\right),
%\end{eqnarray*}
%and 
\begin{eqnarray*}
	B_i &:=& \left(
\begin{array}{cccc}
\lambda_i x^{v^{i1}}v^{i1} & \cdots  & \lambda_i x^{v^{i(r_i-1)}}v^{i(r_i-1)} & \lambda_i x^{v^{ir_i}}v^{ir_i}\\
0                & \cdots  & 0                    & 0                    \\
\vdots           & \cdots  & \vdots               & \vdots              \\
x^{v^{i1}}       & \cdots  &  x^{v^{i(r_i-1)}}    & x^{v^{ir_i}}\\
\vdots           & \cdots  & \vdots               & \vdots              \\
0                & \cdots  & 0                    & 0                    \\
\end{array}\right)
\begin{array}{cccc}
\\
\\
\\
\text{$\leftarrow (n+i)\text{th}$ row}\\
\\
\\
\end{array}
\end{eqnarray*}
for $i = 1,\ldots,p.$ Here, $v^{ij}$ are written as column vectors. 

If $(x, c, \lambda)\in \Psi^{-1}(0)$, we know that
$\lambda_ix^{v^{ij}} \ne  0$ for $i = 1,\ldots,p,$ and $j = 1,\ldots,
r_i.$ Hence, for $i=1, \ldots, p,$ $B_i$ has the same rank as the matrix
\begin{eqnarray*}
	C_i &:=& \left(
\begin{array}{cccc}
v^{i1} & \cdots  & v^{i(r_i-1)}          & v^{ir_i}\\
0                & \cdots  & 0           & 0                    \\
\vdots           & \cdots  & \vdots      & \vdots              \\
1                & \cdots  & 1           & 1\\
\vdots           & \cdots  & \vdots      & \vdots              \\
0                & \cdots  & 0           & 0                    \\
\end{array}\right)
\begin{array}{cccc}
\\
\\
\\
\text{$\leftarrow (n+i)\text{th}$ row}\\
\\
\\
\end{array}
\end{eqnarray*}
%\begin{eqnarray*}
%M_4 &:=& \left (
%\begin{array}{cccc|c|ccccccccccc}
%$$v^{11} & \cdots  & v^{1 (r_1-1)} & v^{1r_1} & \cdots  & v^{p1} & \cdots  & v^{p (r_p-1)} & v^{pr_p}$$\\
%\hline
%$$1      & \cdots  & 1		    & 1        & \cdots  & 0      & \cdots  & 0           & 0       $$\\
%$$0      & \cdots  & 0           & 0        & \cdots  & 0      & \cdots  & 0           & 0       $$\\
%$$\vdots & \cdots  & \vdots      & \vdots   & \cdots  & \vdots & \cdots  & \vdots      & \vdots  $$\\
%$$0      & \cdots  & 0           & 0        & \cdots  & 1      & \cdots  & 1  	        & 1$$\\
%\end{array}\right).\end{eqnarray*}
By subtracting the last column from the first $r_i-1$ columns of $C_i$, we get the following matrix having the same rank as $C_i$:
\begin{eqnarray*}
	D_i &:=& \left(
\begin{array}{cccc}
w^{i1} & \cdots  & w^{i(r_i-1)}          & v^{ir_i}\\
0                & \cdots  & 0           & 0                    \\
\vdots           & \cdots  & \vdots      & \vdots              \\
0                & \cdots  & 0           & 1\\
\vdots           & \cdots  & \vdots      & \vdots              \\
0                & \cdots  & 0           & 0                    \\
\end{array}\right)
\begin{array}{cccc}
\\
\\
\\
\text{$\leftarrow (n+i)\text{th}$ row}\\
\\
\\
\end{array}
\end{eqnarray*}
%By some linear operations on the columns of $M_4$, we produce the
%matrix  
%\begin{eqnarray*}
%M_5 &:=& \left (
%\begin{array}{cccc|c|ccccccccccc}
%$$v^{11}-v^{1r_1} & \cdots  & v^{1 (r_1-1)}-v^{1r_1} & v^{1r_1} & \cdots  & v^{p1}-v^{pr_p} & \cdots  & v^{p (r_p-1)}-v^{pr_p} & v^{pr_p}$$\\
%\hline
%$$0      & \cdots  & 0 	        & 1        & \cdots  & 0      & \cdots  & 0           & 0       $$\\
%$$0      & \cdots  & 0           & 0        & \cdots  & 0      & \cdots  & 0           & 0       $$\\
%$$\vdots & \cdots  & \vdots      & \vdots   & \cdots  & \vdots & \cdots  & \vdots      & \vdots  $$\\
%$$0      & \cdots  & 0           & 0        & \cdots  & 0      & \cdots  & 0  	        & 1$$\\
%\end{array}\right) \\
%&=& \left (
%\begin{array}{cccc|c|ccccccccccc}
%$$w^{11} & \cdots  & w^{1 (r_1-1)} & v^{1r_1} & \cdots  & w^{p1} & \cdots  & w^{p (r_p-1)} & v^{pr_p}$$\\
%\hline
%$$0      & \cdots  & 0		    & 1        & \cdots  & 0      & \cdots  & 0           & 0       $$\\
%$$0      & \cdots  & 0           & 0        & \cdots  & 0      & \cdots  & 0           & 0       $$\\
%$$\vdots & \cdots  & \vdots      & \vdots   & \cdots  & \vdots & \cdots  & \vdots      & \vdots  $$\\
%$$0      & \cdots  & 0           & 0        & \cdots  & 0      & \cdots  & 0  	        & 1$$\\
%\end{array}\right).
%\end{eqnarray*}
It is clear that $M_3$ has the same rank as 
$$M_4:=\left(\begin{array}{c|c|c|c} D_1 & D_2 & \cdots & D_p\end{array}\right).$$
Rearranging the columns of $M_4$ by moving the last column of each $D_i$ to the right of $M_4$, we get
\begin{eqnarray*}
M_5 &:=& \left (
\begin{array}{ccccccc|ccccccccc}
$$w^{11} & \cdots  & w^{1 (r_1-1)} & \cdots  & w^{p1} & \cdots  & w^{p(r_p-1)} & v^{1r_1} & \cdots  & v^{pr_p}$$\\
\hline
$$       &     &  		     &     &        &     &             &          &     &     $$\\
$$       &     &  		     &     &        &     &             & 1        &     & \textbf{\LARGE 0}       $$\\
$$	     &     &             & \textbf{\LARGE 0}&&&             &          & \ddots &  $$\\
$$       &     &             &     &        &     &    	        & \textbf{\LARGE 0}&& 1  $$\\
\end{array}\right).\end{eqnarray*}
In view of~\eqref{Lemma42}, 
%we have
%$$\mathrm{rank}\{w^{11},\ldots,w^{1 (r_1-1)},\ldots, w^{p1},\ldots, w^{p (r_p-1)}\}= d = n.$$
$\mathrm{rank}{M_5}=n+p$ on $\Psi^{-1}(0)$. Thus $D\Psi$ is of maximal rank on $\Psi^{-1}(0)$, namely $\Psi\pitchfork\{0\}$ in $\mathbb{R}^{n+p}$. 
Note that ${U}_1\cap  \cdots \cap {U}_p$ is an open dense
semi-algebraic set in $\mathbb{R}^{\Pi}.$ This, together with 
Theorem~\ref{SardTheorem}, implies that the set 
$$P_2 : =\{c\in  {U}_1\cap  \cdots \cap {U}_p \ : \  \Psi(\cdot, c, \cdot)\pitchfork\{0\}\}$$
contains an open dense semi-algebraic subset of $\mathbb{R}^{\Pi}.$ Since $\Psi(\cdot, c, \cdot) \colon (\mathbb{R}^*)^n\times(\mathbb{R}^p-\{0\})\to\mathbb{R}^{n}\times\mathbb{R}^p$ is a mapping between two manifolds of same dimension, the transversality condition implies that $\Psi(\cdot, c, \cdot)$ is a local diffeomorphism on $(\Psi(\cdot, c, \cdot))^{-1}(0)$ for each $c\in P_2$.

Let $c\in P_2$. If $(\Psi(\cdot, c, \cdot))^{-1}(0)\not=\emptyset$, there exists $(x, \lambda) \in (\mathbb{R}^*)^n\times(\mathbb{R}^p-\{0\})$ such that $\Psi(x, c, \lambda) = 0.$ Note that, for every $t\in {\mathbb R} \setminus \{0\},$ $\Psi(x, c, t\lambda)=0$. So $\Psi(\cdot, c, \cdot)$ is not a local diffeomorphism at $(x, \lambda),$ which is a contradiction. Hence $(\Psi(\cdot, c, \cdot))^{-1}(0)=\emptyset.$ Consequently, $c \in \mathcal{D}_I(\Gamma).$ Therefore, $P_2 \subset \mathcal{D}_I(\Gamma),$ which ends the proof in Case~1.

\subsubsection*{Case 2: $d < n$} 
%Now assume that $d<n$, we need some preparation.   
It is clear that there exist linearly independent vectors $q^1, \ldots, q^{n-d}\in \mathbb{Z}^n$ and numbers $d_{1}, \ldots, d_{n - d} \in \mathbb{R}$ such that the set $\Gamma_1 + \cdots + \Gamma_p$ is contained in the affine space
\begin{eqnarray*}%\label{L}
L &:=& \{\kappa  \in \mathbb{R}^n \ : \ \langle q^j , \kappa  \rangle =d_j,j=1,\ldots,n-d\}.
\end{eqnarray*}
Consequently, we deduce the following lemma.
\begin{lemma}\label{Lemma41}
For each $i \in \{1, \ldots, p\}$ there exist real numbers $d_{ij}, j=1,\ldots, n - d,$ such that $\Gamma_i\subset L_i,$ where 
\begin{equation*}%\label{Li}
L_i := \{\kappa \in \mathbb{R}^n \ : \ \langle   q^j, \kappa \rangle   = d_{ij}, j=1,\ldots,n-d\},
\end{equation*}
i.e., $\Gamma_i$ is contained in an affine space parallel to $L.$
\end{lemma}

For each $j = 1, \ldots, n-d,$ let us write $q^j := (q_{j1}, \ldots, q_{jn}).$
Without loss of generality, suppose that $q_{11}>0.$ For $j = 2, \ldots, n-d,$ by replacing $q^j$ by $q^j+Nq^1$ for $N\in\mathbb Z_+$ large enough, we can assume that $q_{j1}>0.$ 
%In addition, we have the following lemma.
%\begin{lemma} 
Recall that $I=\{1,\dots,p\}$ and $\mathcal {D}_{I}(\Gamma)$ is defined by~\eqref{DIG}. 
For any $N\in \mathbb{Z}_+$, set $$\Gamma_N:=(\Gamma_1+Ne^1, \ldots, \Gamma_p + Ne^1),$$ 
i.e., $\Gamma_N$ is the translation of $\Gamma$ by the vector $N e^1.$
Then $\Gamma$ and $\Gamma_N$ have the same number of integer points. Furthermore,
$$\mathcal {D}_{I}(\Gamma) = \mathcal {D}_{I}(\Gamma_N).$$
Indeed, for all $c\in\mathbb R^{\Pi},$ it is not hard to check that the set
$$\{x\in(\mathbb R^*)^n:\ F(x,c)=0\ \text{ and }\ \rank(xDF(x,c))=p\}$$
is equal to the set
$$\{x\in(\mathbb R^*)^n:\ x_1^N F(x,c)=0\ \text{and}\ \rank(D(x_1^NF)(x,c))=p\}.$$
So the equality holds.
%So the lemma follows.
%\end{proof}
Observe that 
$$\Gamma_i + Ne^1 \subset\{\kappa \in \mathbb{R}^n \ : \ \langle   q^j, \kappa \rangle   = d_{ij}+Nq_{j1}, j=1,\ldots,n-d\}.$$
Hence, by replacing $\Gamma$ by $\Gamma_N$ for $N\in \mathbb Z_+$ large enough, we can suppose that 
\begin{equation*}\label{dij>0}
d_{ij}\geqslant 0\ \text{ for }\ i=1,\dots,p\ \text{ and }\ j=1,\dots,n-d.
\end{equation*}
%In view of Lemma 4.1, we can assume that the convex hull of {0, q^1
%, . . . , q^{n−d} does not contain integer points other than {0, q^1, . . . , q^{n−d}.
In view of Lemma~\ref{Integer}, and for simplicity of notation, we can assume that there are vectors $q^{n-d+1},\dots,q^n\in\mathbb Z^n$ such that $|\det(q^1,\dots,q^n)|=1$ and for $i=1,\dots,p$ and $j=n-d+1,\dots,n,$ we have
\begin{equation}\label{d>0}\langle q^j,\kappa\rangle\geqslant 0\ \text{ for all }\ \kappa \in \Gamma_i.\end{equation}
%Again, replacing $q^j,\ j=n-d+1,\dots,n$ by a non-zero linear combination $\sum_{k=1}^{n}t_kq^k$ with $0\leqslant t_k<1,$ we can assume that the convex hull of $\{0,q^1,\dots,q^{n}\}$ does not contain integer points other than $\{0,q^1,\dots,q^{n}\}$. Then the property~\eqref{d>0} is preserved by this operation. Moreover,
%$$|\det(q^1,\ldots, q^n)| =1.$$ 
%
%\begin{proof} 
%For each index $i = 1, \ldots, p,$ choose $\kappa^i \in \Gamma_i.$ Fix an index $i$ and set $d_{ij} := \langle   q^j, \kappa^i \rangle$ for $j = 1, \ldots, n - d.$ Take any $\kappa \in \Gamma_i$ with $\kappa \ne \kappa^i.$ By definition, $\gamma := \sum_{k = 1}^p \kappa^k$ and $\gamma' := \gamma + (\kappa - \kappa^i)$ belong to the set $\Gamma_1 + \cdots + \Gamma_p.$ We have for all $j = 1, \ldots, n - d,$
%\begin{eqnarray*}
%\langle q^j , \kappa - \kappa^i \rangle &=&   
%\langle q^j , \gamma' - \gamma \rangle \ = \    
%\langle q^j , \gamma' \rangle -  \langle q^j , \gamma \rangle \ = \ d_j - d_j \ = \ 0,
%\end{eqnarray*}
%and so $\langle  q^j, \kappa \rangle = \langle  q^j, \kappa^i \rangle   = d_{ij}.$ Consequently, $\kappa \in L_i.$ The lemma is proved.
%\end{proof}
%For each $j = 1, \ldots, n,$ let us write $q^j := (q_{j1}, \ldots, q_{jn})$ and 
Set $A := (q_{jk})_{j,k=1,\ldots,n}$. Consider the following change of coordinates
\begin{equation}\label{Eqn1}
\left \{
\begin{array}{lcccl}
x_1 &=& u_1^{q_{11}}\ldots u_{n-d}^{q_{(n - d)1}}\ldots u_n^{q_{n1}}, \\
\vdots & \vdots & \vdots \\
x_n &=& u_1^{q_{1n}}\ldots u_{n-d}^{q_{(n - d)n}}\ldots u_n^{q_{nn}}.
\end{array}\right.
\end{equation}
It follows from Lemma~\ref{Lemma41} that for each $\kappa \in \mathcal{Z}_i,$ we have $A \kappa = (d_{i1},\ldots,d_{i (n - d)},
\gamma_\kappa),$ for some $\gamma_\kappa \in \mathbb{Z}_+^d.$ So in the system of coordinates $u_1,\ldots,u_n$, the polynomial $f_i(x,c_i)$ has the form 
\begin{equation}\label{Eqn2}
\dis u_1^{d_{i1}}\ldots u_{n - d}^{d_{i (n - d)}}\sum_{\kappa \in \mathcal{Z}_i}c_{i, \kappa}u'^{\gamma_\kappa}\in\mathbb{R}[u]
\end{equation}
where $u'=(u_{n-d+1},\ldots,u_n).$ Set 
\begin{equation}\label{Eqn3}
g_i(u',c_i) := \dis\sum_{\kappa \in \mathcal{Z}_i}c_{i,\kappa}u'^{\gamma_\kappa}\in\mathbb{R}[u'].
\end{equation}
Since $A$ is an integer matrix and $|\det(A)|=1$, the monomial mapping
\eqref{Eqn1} admits a unique monomial inverse mapping, given by
$A^{-1}$. Hence the system  $f_1(x, c_1) = \cdots = f_p(x, c_p) = 0$ has solutions in $(\mathbb{R}^*)^n$ if and only if the system  $g_1(u',c_1) = \cdots = g_p(u',c_p) = 0$ has solutions in $(\mathbb{R}^*)^d.$ There are two cases to be considered (recall that $ n > d$ and $n \ge p$).

\subsubsection*{Case 2.1: $d < p$} Consider the semi-algebraic mapping
\begin{eqnarray*}
{G} \colon (\mathbb{R}^*)^d  \times \mathbb{R}^{\Pi} \to  \mathbb{R}^p, \quad  (u', c_1,\ldots,c_p) \mapsto (g_1(u', c_1),\ldots,g_p(u', c_p)).
\end{eqnarray*}
For $i = 1, \ldots, p,$ let $\kappa^i \in \mathcal{Z}_i.$
The Jacobian matrix $D {G}$ of ${G}$ contains the following diagonal  matrix
$$\frac{\partial {G}}{\partial (c_{1, \kappa^1},\ldots,c_{p, \kappa^p})} 
\ = \ \begin{pmatrix}
u'^{\gamma_{\kappa^1}} & & \textbf{\large 0} \\
& \ddots & \\
\textbf{\large 0} & & u'^{\gamma_{\kappa^p}}
\end{pmatrix},$$
which has rank $p$ since $u'\in(\mathbb{R}^*)^d.$ Hence $D{G}$ is of
rank $p,$ which yields ${G}\pitchfork \{0\}$ in $\mathbb{R}^p.$ By
Theorem~\ref{SardTheorem}, the set 
$$P_3 := \{c := (c_1,\ldots,c_p) \in \mathbb{R}^{\Pi} \ : \ {G}(\cdot, c)\pitchfork\{0\}\}$$
contains an open dense semi-algebraic subset of $\mathbb{R}^{\Pi}.$ Since $d < p,$ the mapping ${G}(\cdot, c) \colon (\mathbb{R}^*)^d \to \mathbb{R}^p$ is transverse to $\{0\}$ if and only if $\mathrm{Im}  {G}(\cdot, c) \cap \{0\} = \emptyset.$ We deduce, for each $c \in P_3,$ that
$\{{G}(\cdot, c) = 0\}\cap(\mathbb{R}^*)^d=\emptyset,$ and hence  $\{F(\cdot, c) = 0\} \cap (\mathbb{R}^*)^n = \emptyset.$ 
This implies that $P_3 \subset \mathcal{D}_I(\Gamma).$

\subsubsection*{Case 2.2: $d \geqslant p$} We show that this case can be reduced to the case $n = d.$ 
To see this, fix $c := (c_1, \ldots, c_p) \in \mathbb{R}^{\Pi}.$ Under the change of coordinates (\ref{Eqn1}), the polynomials $f_i(x, c_i)\in\mathbb{R}[x]$ and $g_i(u', c_i)\in\mathbb{R}[u']$ have the forms (\ref{Eqn2}) and (\ref{Eqn3}), respectively. 
Recall that $F(x,c)= (f_1(x,c_1),\ldots,f_p(x,c_p))$ and $G(u', c) = (g_1(u',c_1),\ldots,g_p(u',c_p)).$ We have seen that $F(x,c)=0$ has solutions in $(\mathbb{R}^*)^n$ if and only if $G(u',c)=0$ has solutions in $(\mathbb{R}^*)^d.$ 

For any $\kappa\in\mathbb{Z}^n$, let  $\kappa=(\kappa_1,\ldots,\kappa_n)$.  By a direct calculation, in the system of coordinates $u_1,\ldots,u_n$, the matrix $xDF(x,c)$ has the form
\begin{equation}\label{Eqn4}
	\left(u_1^{d_{i1}}\ldots u_{n - d}^{d_{i (n - d)}}\sum_{\kappa \in
		\mathcal{Z}_i}\kappa_lc_{i,
		\kappa}u'^{\gamma_\kappa}\right)_{i=1,\ldots,p,\ l=1,\ldots,n}.
	\end{equation}
Furthermore, let $u'DG(u',c)$ denote the matrix
\[\left(u'_j\frac{\partial g_i}{\partial u'_j}(u',c_i)\right)_{i=1,\ldots,p,\ j=n-d+1,\ldots,n}\]
then we can see that
\begin{eqnarray*}
u'DG(u', c) &=& \left(\sum_{\kappa \in \mathcal{Z}_i}\left(\sum_{l=1}^n q_{j,l}\kappa_l\right)c_{i, \kappa}u'^{\gamma_\kappa}\right)_{i=1,\ldots,p,\ j=n-d+1,\ldots,n}. 
\end{eqnarray*}
Observe that the columns of $u'DG(u',c)$ are linear combinations of the columns of the matrix 
\[
\left(\sum_{\kappa \in \mathcal{Z}_i}\kappa_lc_{i,
		\kappa}u'^{\gamma_\kappa}\right)_{i=1,\ldots,p,\
		l=1,\ldots,n},
\]
which has the same rank as the matrix in \eqref{Eqn4} for any
$u\in(\mathbb{R}^*)^n$. The monomial mapping
\eqref{Eqn1} admits a unique monomial inverse mapping. 
Consequently, we have
$$\{ x \in (\mathbb{R}^*)^n \ : \ F(x,c) =  0 \quad  \textrm{ and } \quad \mathrm{rank}(x D F(x,c)) < p\} \ne \emptyset$$ 
if and only if 
$$\{ u' \in (\mathbb{R}^*)^d \ : \ G(u',c) =  0 \quad  \textrm{ and } \quad \mathrm{rank}(u'DG(u',c)) < p\} \ne \emptyset.$$ 
We note, in addition, the following facts:
\begin{itemize}
\item For each $c_i \in \mathbb{R}^{m_i},$ $g_i(\cdot, c_i)$ is a polynomial function in $d$ variables.
\item Let $\pi\colon\mathbb R^n\to\mathbb R^d$ be the projection on the last $d$ coordinates, and for $i = 1,\dots,n$, we write $A(\Gamma_i)$ for the set
$\{A \kappa \ : \ \kappa \in \Gamma_i\}.$ 
%image of $\Gamma_i$ by the linear mapping whose standard matrix is $A$. 
In light of~\eqref{d>0}, $\pi(A(\Gamma_i))$ is a Newton polyhedron and is equal to the Newton polyhedron of $g_i$. As the matrix $A$ is nonsingular, $\gamma_{\kappa}$ in~\eqref{Eqn2} are distinct for different $\kappa,$ and so $\pi(A(\Gamma_i))$ has the same number of integer points as $\Gamma_i.$ Moreover, we have
$$\dim (\pi(A(\Gamma_1))+\cdots+\pi(A(\Gamma_p)))=\dim(\Gamma_1+\cdots+\Gamma_p)=d.$$
\end{itemize}
Therefore, the problem is reduced to the case $n = d.$ This ends the proof of the proposition.
\end{proof}

\subsection*{Acknowledgments}
The authors wish to thank Professor Huy Vui H\`a for useful discussions and the referee(s) for her/his/their careful reading and constructive comments on the manuscript. 
\thanks{$^{\dag}$The first author is partially supported by Vietnam National Foundation for Science and Technology Development (NAFOSTED) grant 101.04-2019.305.}
\thanks{$^\ddag$The second author was supported by the Chinese National Natural Science Foundation under grants 11401074, 11571350.}
\thanks{$^{*}$The third author is partially supported by Vietnam National Foundation for Science and Technology Development (NAFOSTED), grant 101.04-2019.302.}

\end{document}